\documentclass[reqno, 11pt]{amsart}  %arXiv version

\usepackage{amssymb}

\usepackage{graphicx}
\usepackage[usenames,dvipsnames,svgnames]{xcolor}
\usepackage[shortlabels]{enumitem} % Can use enumerate

\usepackage[colorlinks, urlcolor=Navy, linkcolor=Navy, citecolor=Navy]{hyperref} % If you use hyperref, don't insert indention in \if()...\fi

\usepackage[all]{xy}
\usepackage{tikz}
\usepackage{tikz-cd}
\usetikzlibrary{arrows,shapes,chains,matrix,positioning,scopes,cd,patterns}

\usepackage{todonotes}

\DeclareFontFamily{U}{mathx}{\hyphenchar\font45}
\DeclareFontShape{U}{mathx}{m}{n}{
      <5> <6> <7> <8> <9> <10>
      <10.95> <12> <14.4> <17.28> <20.74> <24.88>
      mathx10
      }{}
\DeclareSymbolFont{mathx}{U}{mathx}{m}{n}
\DeclareFontSubstitution{U}{mathx}{m}{n}
\DeclareMathAccent{\widecheck}{0}{mathx}{"71}
\DeclareMathAccent{\wideparen}{0}{mathx}{"75}

\usepackage[top=30truemm,bottom=30truemm,left=20truemm,right=20truemm]{geometry}

\theoremstyle{plain}
\newtheorem{thm}{Theorem}[section]
\newtheorem{lem}[thm]{Lemma}

\newtheorem{prop}[thm]{Proposition}

\newtheorem{thm*}{Theorem}
\theoremstyle{definition}
\newtheorem{dfn}[thm]{Definition}

\newtheorem{rem}[thm]{Remark}
\newtheorem{que}[thm]{Question}

\numberwithin{equation}{section}

\DeclareMathOperator{\Hom}{Hom}

\DeclareMathOperator{\End}{End}
\DeclareMathOperator{\Ext}{Ext}

\renewcommand{\mod}{\mathop{\mathsf{mod}}\nolimits}
\DeclareMathOperator{\Mod}{\mathsf{Mod}}

\DeclareMathOperator{\add}{\mathsf{add}}

\DeclareMathOperator{\gen}{\mathsf{gen}}
\DeclareMathOperator{\proj}{\mathsf{proj}}

\DeclareMathOperator{\ind}{\mathsf{ind}}
\DeclareMathOperator{\Sim}{\mathsf{sim}}
\DeclareMathOperator{\Fac}{\mathsf{gen}}

\DeclareMathOperator{\sub}{\mathsf{sub}}

\newcommand{\xto}[1]{\xrightarrow{#1}}

\DeclareMathOperator{\Cok}{Cok}
\renewcommand{\Im}{\mathop{\mathrm{Im}}}

\DeclareMathOperator{\irr}{irr}

\DeclareMathOperator{\projdim}{proj\text{.}dim}

\DeclareMathOperator{\Dim}{\underline{dim}}

\DeclareMathOperator{\rank}{rank}
\DeclareMathOperator{\urank}{\underline{\rank}}

\DeclareMathOperator{\Spec}{Spec}

\newcommand{\supp}{\operatorname{Supp}\nolimits}

\DeclareMathOperator{\GL}{GL}
\DeclareMathOperator{\diag}{diag}
\DeclareMathOperator{\M}{M}

\newcommand{\ov}[1]{\overline{#1}}

\renewcommand{\epsilon}{\varepsilon}

\renewcommand{\L}{\Lambda}

\newcommand{\bP}{\mathbb{P}}

\newcommand{\bT}{\mathbb{T}}

\newcommand{\bZ}{\mathbb{Z}}

\newcommand{\m}{\mathfrak{m}}

\newcommand{\p}{\mathfrak{p}}
\newcommand{\q}{\mathfrak{q}}

\newcommand{\cA}{\mathcal{A}}
\newcommand{\cB}{\mathcal{B}}

\newcommand{\cH}{\mathcal{H}}
\newcommand{\cI}{\mathcal{I}}

\newcommand{\cO}{\mathcal{O}}
\newcommand{\cP}{\mathcal{P}}

\newcommand{\cR}{\mathcal{R}}
\newcommand{\cS}{\mathcal{S}}
\newcommand{\cT}{\mathcal{T}}
\newcommand{\cU}{\mathcal{U}}

\newcommand{\cX}{\mathcal{X}}
\newcommand{\cY}{\mathcal{Y}}

\newcommand{\sD}{\mathsf{D}}

\newcommand{\sK}{\mathsf{K}}

\newcommand{\sP}{\mathsf{P}}

\newcommand{\sT}{\mathsf{T}}

\newcommand{\tors}{\mathop{\mathsf{tors}}}
\newcommand{\torf}{\mathop{\mathsf{torf}}}
\newcommand{\ftors}{\mathop{\mathsf{f}\text{-}\mathsf{tors}}}
\newcommand{\ntors}{\mathop{\mathsf{nf}\text{-}\mathsf{tors}}}

\newcommand{\excep}{\mathop{\mathsf{excep}}}
\newcommand{\brick}{\mathop{\mathsf{brick}}}

\newcommand{\twosilt}{\mathop{\mathsf{2}\text{-}\mathsf{silt}}}

\newcommand{\Clus}{\mathop{\mathsf{Clus}}}

\begin{document}
\title[Torsion classes of extended Dynkin quivers over commutative rings]{Torsion classes of extended Dynkin quivers \\ over commutative rings}
\author[O. Iyama]{Osamu Iyama}
\author[Y. Kimura]{Yuta Kimura}
\address{Osamu Iyama : Graduate School of Mathematical Sciences, The University of Tokyo, 3-8-1 Komaba Meguro-ku Tokyo 153-8914, Japan}
\email{iyama@ms.u-tokyo.ac.jp}
\address{Yuta Kimura: Osaka Central Advanced Mathematical Institute, Osaka Metropolitan University, 3-3-138 Sugimoto, Sumiyoshi-ku Osaka 558-8585, Japan}
\email{yutakimura@omu.ac.jp (Previous)}
%\curraddr{Department of Mechanical Engineering and Informatics, Faculty of Engineering, Hiroshima Institute of Technology, 2-1-1 Miyake, Saeki-ku Hiroshima 731-5143, Japan}
\email{y.kimura.4r@cc.it-hiroshima.ac.jp (Current)}
%\subjclass[2010]{Primary 16G30, 16E05}
\begin{abstract}
For a Noetherian $R$-algebra $\L$, there is a canonical inclusion $\tors\L\to\prod_{\p\in\Spec R}\tors(\kappa(\p)\L)$, and each element in the image satisfies a certain compatibility condition. We call $\L$ \emph{compatible} if the image coincides with the set of all compatible elements. For example, for a Dynkin quiver $Q$ and a commutative Noetherian ring $R$ containing a field, the path algebra $RQ$ is compatible. In this paper, we prove that $RQ$ is compatible when $Q$ is an extended Dynkin quiver and $R$ is either a Dedekind domain or a Noetherian semilocal normal ring of dimension two.
\end{abstract}
\maketitle
\tableofcontents
%%%%%%%%%%%%%%%%%%%%%%%%%%%%%%%%%%%%%%%%%%%%%%%%%%
\section{Introduction}\label{section-introduction}
Let $R$ be a commutative Noetherian ring and $Q$ a finite acyclic quiver.
The aim of this paper is to develop the classification problem of torsion classes of $\mod RQ$: the category of finitely generated modules over the path algebra $RQ$.
Such algebra is a kind of \emph{Noetherian algebras} (also called \emph{module-finite $R$-algebras}), that is, an $R$-algebra which is finitely generated as an $R$-module.
Tilting and silting theory of Noetherian algebras have been widely studied \cite{CB, Gabriel, Gnedin, Iyama-Reiten, Iyama-Wemyss-maximal, Kimura, Stanley-Wang, Takahashi}.
There are several results for classification problems of torsion classes and tilting modules of $kQ$ over a field $k$ \cite{AnSa, Buan}.

Let $\L$ be a Noetherian $R$-algebra.
A subcategory $\cT$ of $\mod\L$ is called a \emph{torsion class} if it is closed under taking factor modules and extensions.
We denote by $\tors\L$ the set of all torsion classes of $\mod\L$.
When $\L=R$, there exists a Gabriel-type bijection between $\tors R$ and the set of specialization closed subsets of $\Spec R$ \cite{Gabriel, Stanley-Wang}. For a Noetherian algebra $\L$, the structure of $\tors\L$ is much richer, and often categorifies important combinatorial structure such as cluster algebras \cite{Ingalls-Thomas} and Coxeter groups \cite{DIRRT}.

For a prime ideal $\p$ of $R$, let $\kappa(\p)=R_\p/\p R_\p$, $\L_\p=R_\p\otimes_R \L$. In \cite[Theorem 3.16]{Iyama-Kimura}, the authors constructed a canonical injective map
\[
\Phi_{\rm t} : \tors\L \longrightarrow \bT_R(\L):=\prod_{\p\in\Spec R}\tors(\kappa(\p)\otimes_R\L), \quad \cT \mapsto (\kappa(\p)\otimes_R\cT)_{\p\in\Spec R},
\]
where $\kappa(\p)\otimes_R\cT:=\{\kappa(\p)\otimes_R X \mid X\in\cT\}$ is a torsion class of $\mod(\kappa(\p)\otimes_R\L)$.
Therefore it is important to give an explicit description of the image of this map.
Let us recall a necessary condition for elements in $\bT_R(\L)$ to belong to $\Im\Phi_{\rm t}$: 
For prime ideals $\p\supseteq \q$ of $R$, we have an order preserving map
\begin{eqnarray*}
&{\rm r}_{\p\q} : \tors(\kappa(\p)\otimes_R\L) \longrightarrow \tors(\kappa(\q)\otimes_R\L), \quad \cT \mapsto \kappa(\q)\otimes_{R_\p}\psi_\p(\cT), \mbox{ where}&\\
&\psi_\p(\cT):=\{X \in \mod\L_\p \mid \kappa(\p)\otimes_{R_\p}X \in \cT\}.&
\end{eqnarray*}
An element $(\cX^\p)_\p$ of $\bT_R(\L)$ is called \emph{compatible} if ${\rm r}_{\p\q}(\cX^\p) \supseteq \cX^\q$ holds for any prime ideals $\p\supseteq \q$ of $R$. Then any element of $\Im\Phi_{\rm t}$ is compatible \cite[Proposition 3.20]{Iyama-Kimura}.
We say that a Noetherian algebra is \emph{compatible} if the converse holds, that is, $\Im\Phi_{\rm t}$ coincides with the set of all compatible elements of $\bT_R(\L)$.
In this case, $\tors\L$ can be described in terms of torsion classes of a family $(\kappa(\p)\otimes_R\L)_{\p\in\Spec R}$ of finite dimensional algebras
\[
\tors\L \simeq \left\{(\cX^\p)_\p\in \bT_R(\L)  \,\middle|\,\ {\rm r}_{\p\q}(\cX^\p) \supseteq \cX^\q,\,{}^{\forall}\p \supseteq \q \in \Spec R\right\}.
\]
Therefore the following question posed in \cite{Iyama-Kimura} is important.

\begin{que}
Characterize Noetherian algebras which are compatible.
\end{que}

Some classes of Noetherian algebras are known to be compatible. For example, if $R$ is a semilocal ring of dimension one, then any Noetherian $R$-algebra is compatible \cite[Theorems 1.5]{Iyama-Kimura}.
Also, if $R$ is a commutative Noetherian ring $R$ containing a field $k$, then the $R$-algebra $R\otimes_kA$ is compatible for each finite dimensional $k$-algebra $A$ which is $\tau$-tilting finite \cite[Corollary 1.8]{Iyama-Kimura}.

On the other hand, for a Dynkin quiver $Q$ and a commutative Noetherian ring $R$, it is known that $RQ$ is compatible and there exists an isomorphism of posets \[ \tors RQ \simeq \Hom_{\rm poset}(\Spec R, \Clus(Q)) \]
by \cite[Theorem 4.1]{IK2}, where $\Clus(Q)$ is the poset of clusters of $Q$ \cite[Definition 3.1]{IK2}.
This generalizes Ingalls-Thomas correspondence \cite{Ingalls-Thomas} when $R$ is a field.
The aim of this paper is to prove a similar result for extended Dynkin quivers $Q$.
Precisely speaking, we prove the following result, where for $n\ge0$, we say that $R$ satisfies ($R_n$) if $R_\p$ is regular for each prime ideal $\p$ of $R$ with height at most $n$ \cite{Bruns-Herzog}.

\begin{thm}\label{thm-compatible-extDynkin}
Let $R$ be a commutative Noetherian ring and satisfying the following conditions.
\begin{enumerate}[\rm(I)]
\item $R$ satisfies \emph{($R_1$)}.
\item $R$ has only finitely many prime ideals of height $2$.
\end{enumerate}
Then, for each extended Dynkin quiver $Q$, the $R$-algebra $RQ$ is compatible.
%, that is, any compatible element of $\bT_R(RQ)$ belongs to the image of $\Phi_{\rm t} : \tors RQ \to \bT_R(RQ)$.
Thus we have an isomorphism of posets
\[\tors RQ \simeq \left\{(\cX^\p)_\p\in \bT_R(RQ)  \,\middle|\, {\rm r}_{\p\q}(\cX^\p) \supseteq \cX^\q,\,{}^{\forall}\p \supseteq \q \in \Spec R\right\},\ \cT\mapsto(\kappa(\p)\otimes_R\cT)_{\p\in\Spec R}.\]
\end{thm}

The condition (II) implies $\dim R\le 2$ by Krull's Hauptidealsatz,  \cite[Theorem 144]{Kaplansky}.
For example, the conditions (I) and (II) are satisfied by the following two classes of commutative Noetherian rings.
\begin{enumerate}[\rm(a)]
\item Dedekind domains.
\item Semilocal normal rings of dimension $2$.
\end{enumerate}
To prove Theorem \ref{thm-compatible-extDynkin}, we give %several observations on $\tors RQ$. We give 
the following characterizations of functorially finite torsion classes of $\mod kQ$ for a field $k$, which is interesting by itself.
We say that $\cT\in\tors kQ$ is \emph{upper finite} (respectively, \emph{lower finite}) if an interval $[\cT, \mod kQ]$ (respectively, $[0, \cT]$) in $\tors kQ$ is a finite set.
We denote by $\cP$ (respectively, $\cR, \cI$) the subcategory of preprojective (respectively, regular, preinjective) $kQ$-modules, by $\ftors kQ$ the set of all functorially finite torsion classes of $\mod kQ$, and by $\ntors kQ:=\tors kQ \setminus \ftors kQ$ the set of all non-functorially finite torsion classes of $\mod kQ$, see subsection \ref{subsection-fft} for details.

\begin{thm}[Theorem \ref{thm-ch-ftors-ext-Dynkin}]
Let $Q$ be an extended Dynkin quiver and $k$ a field.
For $\cT\in\tors kQ$, the following conditions are equivalent.
\begin{enumerate}[{\rm (i)}]
	\item $\cT\in\ftors kQ$.
	\item $\cT$ is either upper finite or lower finite.
	\item Either $\cT\cap\cP\neq 0$ or $\cT^{\perp}\cap\cI\neq 0$ holds (see Conventions below for $(-)^\perp$).
	\item $\cT$ is not contained in an interval $[\cI, \cI\vee \cR]$ in $\tors kQ$.
\end{enumerate}
In particular, we have $\ntors kQ = [\cI, \cI \vee \cR]$ and $\tors kQ = \ftors kQ \sqcup [\cI, \cI \vee \cR]$.
\end{thm}

We also prove the following result, where $\cH$ is the full subcategory of $\cR$ consisting of all homogeneous tubes.

\begin{thm}[Theorem \ref{upper not lower}]
Let $Q$ be an extended Dynkin quiver and $k$ a field, and $\cT\in\tors kQ$.
\begin{enumerate}[\rm(a)]
\item $\cT$ is upper finite if and only if $\cT\cap\cP\neq 0$. Moreover $\cT$ is upper finite and lower infinite if and only if $\ind(\cT\cap\cP)$ is an infinite set if and only if $\cT\cap\cP\neq0$ and $\cH\vee\cI\subset\cT$.
\item $\cT$ is lower finite if and only if $\cT^\perp\cap\cI\neq 0$. Moreover $\cT$ is lower finite and upper infinite if and only if $\ind(\cT^\perp\cap\cI)$ is an infinite set if and only if $\cT^\perp\cap\cI\neq0$ and $\cP\vee\cH\subset\cT^\perp$.
\end{enumerate}
\end{thm}

In section \ref{section-pre-Noe}, we give preliminary results on Noetherian algebras.
Among others, methods of reverse filtrations (e.g.\ Proposition \ref{prop-filt-eq}) play an important role in this paper.
In section \ref{section-path-Noe}, we give preliminaries on path algebras over Noetherian algebras, including a family of modules corresponding to real Schur roots \eqref{prop-cluster-stilt}, and derived Auslander-Reiten translations.
In subsection \ref{subsection-rpq-PRI}, we prove several properties of the map ${\rm r_{\p\q}}$, and prove Theorem \ref{thm-compatible-extDynkin} in subsection \ref{subsec-proof}.
In section \ref{section-tors-Kro}, we give an explicit description of $\tors RQ$ for the path algebra $RQ$ of the Kronecker quiver $Q$ over a Dedekind domain $R$.

\medskip
\noindent{\bf Conventions.} 
Let $\cA$ be an additive category.
For a subcategory  $\cB$ (or a collection of objects) of $\cA$, we denote by $\add\cB$ the full subcategory of $\cA$ consisting of direct summands of finite direct sums of objects in $\cB$.
Let $\cB^{\perp}=\{X\in\cA \mid \Hom_{\cA}(B,X)=0,\,{}^{\forall}B\in\cB\}$.
Similarly, ${}^{\perp}\cB$ is defined.
We denote by $\ind\cA$ the set of isomorphism classes of indecomposable objects in $\cA$.

Let $\cA$ be an abelian category and $\cB$ a subcategory (or a collection of objects) of $\cA$.
We denote by $\gen\cB$ the full subcategory of $\cA$ consisting of factor objects of objects in $\add\cB$.
We denote by $\Sim\cA$ the set of isomorphism classes of simple objects in $\cA$.

%For a posets $X, Y$, we denote by $\Hom_{\rm poset}(X, Y)$ the set of all morphisms of posets, that is, a map $f : X \to Y$ such that $a\leq b$ in $X$ always implies $f(a)\leq f(b)$.
An order preserving map $f : X \to Y$ between posets is called an \emph{embedding of posets} if $f$ is injective and a partial order on $X$ coincides with the restriction of a partial order on $Y$.
Namely, $f$ is an embedding of posets if and only if $f(a) \leq f(b)$ implies $a\leq b$ for any $a, b\in X$.

For a set $X$, we denote by $\sP(X)$ the power set of $X$.
%%%%%%%%%%%%%%%
%%%%%%%%%%%%%%%
\section{Path algebras of extended Dynkin quivers}
%%%%%%%%%%%%%%%
Throughout this section, let $Q$ be an extended Dynkin quiver and $k$ a field.
%%%%%%%%%%%%%%%
%%%%%%%%%%%%%%%
\subsection{Preliminaries on extended Dynkin quivers}
%%%%%%%%%%%%%%%
Let $Q$ be an extended Dynkin quiver and $k$ be a field.
We denote by
\[\tau:\mod kQ\to\mod kQ\]
the Auslander-Reiten translation, and by
\[\cP\ \mbox{ (respectively, $\cR, \cI$)}\]
the subcategory of preprojective (respectively, regular, preinjective) $kQ$-modules.
For subcategories $\cA$ and $\cB$ of $\mod kQ$, let $\cA\vee \cB:=\add(\cA, \cB)$. Then $\mod kQ=\cP\vee\cR\vee\cI$ holds.
Moreover $\Hom_{kQ}(\cR,\cP)=0$, $\Hom_{kQ}(\cI,\cR)=0$ and $\Hom_{kQ}(\cI,\cP)=0$ hold.

The category $\cR$ of regular $kQ$-modules is an extension closed abelian subcategory of $\mod kQ$.
A regular module is called a \emph{regular simple module} if it has no proper regular submodule. For a module $M\in\cR$, the \emph{regular length} of $M$ is a length of composition series of $M$ with respect to regular simple modules.

By \cite[Theorems 3.5 and 4.1]{Dlab-Ringel}, the category $\cR$ decomposes into tubes.
We denote by $\cH$ a product of homogeneous (i.e. rank one) tubes, and denote by $\cR_{\geq 2}$ a product of tubes having ranks $\geq 2$. Then we have
%\begin{equation}\label{R=Hv2}
\[\cR = \cH \vee \cR_{\geq 2}.\]
%\end{equation}
Moreover $\cR_{\geq 2}$ is a product of at most three tubes \cite[Theorem 4.1]{Dlab-Ringel}, and $\Hom_{kQ}(\cH,\cR_{\geq 2})=0$ and $\Hom_{kQ}(\cR_{\geq 2},\cH)=0$ hold. 
More explicitly, we have a decomposition of an abelian category
\begin{equation}\label{R=VRx}
\cR=\bigvee_{x\in|\bP_k^1|}\cR_x,
\end{equation}
where $|\bP_k^1|$ is the set of the closed point of the projective line $\bP_k^1$ over $k$.
%, that is, homogeneous prime ideals $\p$ of the polynomial algebra $S:=k[X,Y]$ with $\deg X=\deg Y=1$ which is neither $(X,Y)$ nor $0$.
Let $h_x$ be the number of isomorphism classes of simple objects in $\cR_x$.
Then we have $\cH=\bigvee_{x\in|\bP_k^1|_1}\cR_x$ and $\cR_{\ge2}=\bigvee_{x\in|\bP_k^1|_{\ge2}}\cR_x$, where $|\bP_k^1|_1$ (respectively, $|\bP_k^1|_{\ge2}$) is the set of $x\in|\bP_k^1|$ with $h_x=1$ (respectively, $h_x\ge2$). Moreover, the Auslander-Reiten quiver of $\cR$ is a disjoint union of those of $\cR_x$'s, which are tubes of rank $h_x$.
Let $\cO_x$ be the local ring of the structure sheaf $\cO_{\bP_k^1}$ at $x$, and $\m_x$ its maximal ideal.
%discrete valuation rings given as the degree zero part of the graded localization of $S$ with respect to $\p$, and $\p_{(\p)}$ the maximal ideal of $S_{(\p)}$.
Then, for a hereditary $\cO_x$-order in $M_{h_x}(\cO_x)$, we have an equivalence
\[\cR_x\simeq\mod_0\left[\begin{smallmatrix}
\cO_x&\cO_x&\cdots&\cO_x\\
\m_x&\cO_x&\cdots&\cO_x\\
\vdots&\vdots&\ddots&\vdots\\
\m_x&\m_x&\cdots&\cO_x
\end{smallmatrix}\right],\]
%Then, for a hereditary $S_{(\p)}$-order in $M_{h_x}(S_{(\p)})$, we have an equivalence
%\[\cR_x\simeq\mod_0\left[\begin{smallmatrix}
%S_{(\p)}&S_{(\p)}&\cdots&S_{(\p)}\\
%\p_{(\p)}&S_{(\p)}&\cdots&S_{(\p)}\\
%\vdots&\vdots&\ddots&\vdots\\
%\p_{(\p)}&\p_{(\p)}&\cdots&S_{(\p)}
%\end{smallmatrix}\right],\]
where $\mod_0$ denotes the category of finite length modules.

To clarify the base field $k$, we use the symbols $\cP^k$, $\cR^k$, $\cI^k$, $\cH^k$ and $\cR_{\geq 2}^k$.
%%%%%%%%%%%%%%%
\subsection{Characterizations of functorially finite torsion classes}\label{subsection-fft}
%%%%%%%%%%%%%%%
In this subsection we show Theorem \ref{thm-ch-ftors-ext-Dynkin}.

Let $A$ be a finite dimensional algebra over a field $k$. 
A torsion class $\cT$ of $\mod A$ is said to be \emph{functorially finite} if for any $kQ$-module $X$, there exists a morphism $f : X \to T$ with $T\in\cT$ such that any morphism from $X$ to a module in $\cT$ factors through $f$.
We denote by $\ftors A$ the set of all functorially finite torsion classes of $\mod A$, and by $\ntors A:=\tors A\setminus \ftors A$ the set of all non-functorially finite torsion classes of $\mod A$.
It is known that there exists a bijection from the set of isomorphism classes of support $\tau$-tilting $A$-modules to $\ftors A$ given by $M\mapsto \gen M$, see \cite{Adachi-Iyama-Reiten} for details. We call $X\in\mod A$ a \emph{brick} if $\End_A(X)$ is a division algebra. We denote by $\brick A$ the set of isomorphism classes of all bricks of $A$. For each $\cT\in\tors A$, the following equality holds \cite[Lemma 3.10]{DIRRT}:
\begin{equation}\label{brick generate tors}
\sT(\cT\cap\brick A)=\cT.
\end{equation}
The set $\tors kQ$ is a poset by inclusion. For $\cT\in\tors kQ$, 
if an interval $[\cT, \mod kQ]$ (respectively, $[0, \cT]$) in $\tors kQ$ is a finite set, $\cT$ is called \emph{upper finite} (respectively, \emph{lower finite}). Otherwise $\cT$ is called \emph{upper infinite} (respectively, \emph{lower infinite}).
For example, $\cI$ and $\cR\vee \cI$ are both upper infinite and lower infinite.

The following theorem gives various characterizations of functorially finite torsion classes of path algebras of extended Dynkin quivers.
%%%%%%%%%%%%%%%
\begin{thm}\label{thm-ch-ftors-ext-Dynkin}
Let $Q$ be an extended Dynkin quiver and $k$ a field.
For $\cT\in\tors kQ$, the following conditions are equivalent.
\begin{enumerate}[{\rm (i)}]
	\item $\cT\in\ftors kQ$.
	\item $\cT$ is either upper finite or lower finite.
	\item Either $\cT\cap\cP\neq 0$ or $\cT^{\perp}\cap\cI\neq 0$ holds.
	\item $\cT$ is not contained in an interval $[\cI, \cI\vee \cR]$ in $\tors kQ$.
\end{enumerate}
In particular, we have
\[
\tors kQ = \ftors kQ \sqcup [\cI, \cI \vee \cR].
\]
\end{thm}
%%%%%%%%%%%%%%%
\begin{rem}
We have an explicit description of the interval $[\cI,\cI\vee\cR]$. In fact, $[\cI, \cI\vee\cR]$ is a \emph{wide interval} in the sense that $\cI^{\perp} \cap (\cI \vee \cR) =\cR$ is a wide subcategory of $\mod kQ$. Thus by \cite[Theorem 4.2]{AP} and \eqref{R=VRx}, we have isomorphisms of posets
\[[\cI, \cI\vee\cR] \simeq \tors\cR\simeq\prod_{x\in|\bP_k^1|}\tors \cR_x\simeq\sP(|\bP_k^1|_1)\times\prod_{x\in|\bP_k^1|_{\ge2}}\tors \cR_x.\]
A description of $\tors\cR_x$ for $x\in|\bP_k^1|_{\ge2}$ can be found in \cite[Theorem 6.4]{Kimura}.
\end{rem}

The proof of Theorem \ref{thm-ch-ftors-ext-Dynkin} is given after Proposition \ref{prop-tors-divid}.

We often use a fact that a $kQ$-module $X$ satisfying $X\simeq \tau X$ is sincere, since this isomorphism implies that the dimension vector of $X$ is an imaginary root of $Q$ \cite[Lemmas VII.4.2, VII.4.11]{ASS}.
%%%%%%%%%%%%%%%
\begin{lem}\label{lem-P-int-fin}
The following statements hold.
\begin{enumerate}[{\rm (a)}]
	\item If $M\in\ind\cP$, then $\ind(M^{\perp})$ is a finite set.
	\item If $M\in\ind\cI$, then $\ind({}^{\perp}M)$ is a finite set.
\end{enumerate}
\end{lem}
%%%%%
\begin{proof}
Since (b) is a dual of (a), we only prove (a).
Take $i \in Q_0$ and $\ell \in\bZ_{\geq 0}$ such that $M \simeq \tau^{-\ell}(kQ e_i)$.
If $\ell=0$, then $M^{\perp}\simeq \mod kQ/\langle e_i \rangle$. Since $kQ/\langle e_i \rangle$ is representation finite, the assertion follows.

Assume $\ell >0$.
For $X\in\ind(M^{\perp})$, we have $0=\Hom_{kQ}(M, X) \simeq \Hom_{kQ}(kQ e_i, \tau^{\ell}X)$. Thus $\tau^{\ell}X$ is not sincere.
Since all but finitely many modules in $\ind\cP$ (respectively, $\ind\cI$) are sincere \cite[Proposition IX.5.6]{ASS}, $\ind(\cP \cap M^{\perp})$ (respectively, $\ind(\cI \cap M^{\perp})$) is finite.
Assume that $X$ belongs to a tube of rank $r\geq 1$.
If $r=1$, then $\tau X \simeq X$ holds.
Thus $\tau^{\ell}X\simeq X$ is sincere, a contradiction.
Assume $r\geq 2$.
If the regular length of $X$ is greater than or equal to $r$, then $X$ is sincere.
Since the number of tubes of rank $r\geq 2$ is finite, $\ind(\cR \cap M^{\perp})$ is finite.
Therefore $\ind(M^{\perp})$ is a finite set.
\end{proof}
%%%%%%%%%%%%%%%
\begin{prop}\label{prop-up-down-finite}
For a torsion class $\cT\in\tors kQ$, the following statements hold.
\begin{enumerate}[{\rm (a)}]
	\item $\cT$ is upper finite if and only if $\cT\cap\cP\neq 0$.
	\item $\cT$ is lower finite if and only if $\cT^{\perp}\cap\cI\neq 0$.
\end{enumerate}
\end{prop}
%%%%%
\begin{proof}
Since (b) is a dual of (a), we only prove (a).
To prove the ``if'' part, let $M\in\ind(\cT\cap\cP)$.
The map $(-)^{\perp}$ induces a bijection from $[\gen M, \mod kQ]$ to $[0, M^{\perp}]\subseteq\torf kQ$.
Thus $[\gen M, \mod kQ]$ is a finite set by Lemma \ref{lem-P-int-fin}(a).
Therefore $\cT$ is upper finite.
To prove the ``only if'' part, assume $\cT \cap \cP=0$.
Then $\cT \subseteq \cR \vee \cI$ holds, and we have $\sharp[\cT, \mod kQ] > \sharp[\cR\vee \cI, \mod kQ] = \infty$.
\end{proof}
%%%%%%%%%%%%%%%
For a $kQ$-module $M$, we denote by $M=M^{\rm pp}\oplus M^{\rm rg}\oplus M^{\rm pi}$ a decomposition with respect to $\cP\vee \cR\vee \cI$, that is, $M^{\rm pp}\in\cP$, $M^{\rm rg}\in\cR$ and $M^{\rm pi}\in\cI$.
We denote by $|M|$ the number of isomorphism classes of indecomposable direct summands of $M$.

We refer the following observation.
%%%%%%%%%%%%%%%
\begin{prop}\label{prop-ftors-fin-int}
For $\cT\in\tors kQ$, the following statements are equivalent.
\begin{enumerate}[{\rm (i)}]
	\item $\cT\in\ftors kQ$.
	\item $\cT$ is either upper finite or lower finite.
\end{enumerate}
\end{prop}
%%%%%
\begin{proof}
(i) implies (ii):
It is enough to show that either $\ind(\cT^{\perp})$ or $\ind\cT$ is a finite set.
Let $\cT=\Fac M$ for a support $\tau$-tilting $kQ$-module $M$.
If $M$ is not sincere, then there exists a proper full subquiver $Q'$ of $Q$ such that $M\in\mod kQ'$.
Since $Q'$ is Dynkin and $\cT\subseteq\mod kQ'$, $\ind\cT$ is a finite set.

Assume that $M$ is sincere, so $M$ is a $\tau$-tilting module.
Since $kQ$ is hereditary, $M$ is a tilting module and hence $|M|=\sharp Q_0$.
Take a decomposition $M=M^{\rm pp}\oplus M^{\rm rg}\oplus M^{\rm pi}$.
We have $|M^{\rm rg}| \leq \sharp Q_0 -2$ by \cite[Lemma XVII.3.4]{Simson-Skowronski}, thus at least one of $M^{\rm pp}$ and $M^{\rm pi}$ is non-zero.
If $M^{\rm pp}\neq 0$, then $\cT=\gen M$ is upper finite by Proposition \ref{prop-up-down-finite}(a).
If $M^{\rm pi}\neq 0$, then $0\neq \tau(M^{\rm pi})\in\sub(\tau M)=\cT^{\perp}$ and hence $\cT^{\perp}\cap\cI\neq 0$.
Thus $\cT$ is lower finite by Proposition \ref{prop-up-down-finite}(b).

(ii) implies (i):
This follows from \cite[Theorem 3.1]{DIJ}.
\end{proof}
%%%%%%%%%%%%%%%
We complete the proof of Theorem \ref{thm-ch-ftors-ext-Dynkin} by the following proposition.
%%%%%%%%%%%%%%%
\begin{prop}\label{prop-tors-divid}
We have $\tors kQ = \ftors kQ \sqcup [\cI, \cI \vee \cR]$.
\end{prop}
%%%%%
\begin{proof}
Let $\cT\in\tors kQ\setminus\ftors kQ$.
By Proposition \ref{prop-ftors-fin-int}, $\cT$ is upper infinite and lower infinite. By Proposition \ref{prop-up-down-finite}, $\cT\cap\cP=0$ and $\cT^\perp\cap\cI=0$.
Thus $\cT\subseteq\cR \vee \cI$ and $\cT^{\perp} \subseteq \cP\vee \cR$. Thus $\cI \subset \cT\subset\cR \vee \cI$.

It remains to show $\ftors kQ \cap [\cI, \cI \vee \cR]=\emptyset$.
Fix $\cT\in \ftors kQ$. By Proposition \ref{prop-ftors-fin-int}, $\cT$ is either upper finite or lower finite. Since $\cI$ and $\cR\vee\cI$ are upper infinite and lower infinite, we have $\cT\notin[\cI, \cI \vee \cR]$, as desired.
\end{proof}

We are ready to prove Theorem \ref{thm-ch-ftors-ext-Dynkin}.

\begin{proof}[Proof of Theorem \ref{thm-ch-ftors-ext-Dynkin}]
(i)$\Leftrightarrow$(ii) follows from by Proposition \ref{prop-ftors-fin-int}. (ii)$\Leftrightarrow$(iii) follows from Proposition \ref{prop-up-down-finite}.
(i)$\Leftrightarrow$(iv) follows from Proposition \ref{prop-tors-divid}.
\end{proof}
%%%%%%%%%%%%%%%
We also prove the following result, which will play a key role in this paper.
%We often use a fact that, if $\cT\in\ftors kQ$ is lower infinite, then there is an infinite strictly decreasing chain $\cT = \cT^0 \supset \cT^1 \supset \cT^2 \supset \cdots$ in $\ftors kQ$, see \cite[Theorem 3.1,\ Proof of Proposition 3.9(d)$\Rightarrow$(b)]{DIJ}.

\begin{thm}\label{upper not lower}
Let $Q$ be an extended Dynkin quiver and $k$ a field, and $\cT\in\tors kQ$.
\begin{enumerate}[\rm(a)]
\item $\cT$ is upper finite if and only if $\cT\cap\cP\neq 0$. Moreover, the following conditions are equivalent.
\begin{enumerate}[\rm(i)]
\item $\cT$ is upper finite and lower infinite.
\item $\ind(\cT\cap\cP)$ is an infinite set.
\item $\cT\cap\cP\neq0$ and $\cH\vee\cI\subset\cT$.
\end{enumerate}
\item $\cT$ is lower finite if and only if $\cT^\perp\cap\cI\neq 0$. Moreover, the following conditions are equivalent.
\begin{enumerate}[\rm(i)]
\item $\cT$ is lower finite and upper infinite.
\item $\ind(\cT^\perp\cap\cI)$ is an infinite set.
\item $\cT^\perp\cap\cI\neq0$ and $\cP\vee\cH\subset\cT^\perp$.
\end{enumerate}
\end{enumerate}
\end{thm}

We start with proving the following result.

\begin{lem}\label{lem-tors-pp-inf}
Assume that $\cT\in\ftors kQ$ is lower infinite. Then the following assertions hold.
\begin{enumerate}[\rm(a)]
\item There is an infinite strictly decreasing chain $\cT = \cT^0 \supset \cT^1 \supset \cT^2 \supset \cdots$ in $\ftors kQ$.
\item $\ind(\cT\cap \cP)$ is an infinite set.
\end{enumerate}
\end{lem}
%%%%%
\begin{proof}
(a) This is \cite[Theorem 3.1,\ Proof of Proposition 3.9(d)$\Rightarrow$(b)]{DIJ}.

(b) By (a), there is an infinite strictly decreasing chain $\cT = \cT^0 \supset \cT^1 \supset \cT^2 \supset \cdots$ in $\ftors kQ$.
Then $\cT^i$ is lower infinite for each $i$. Thus $\cT^i$ is upper finite by Theorem \ref{thm-ch-ftors-ext-Dynkin}(i) and (ii).
Thus $\cT^i \cap \cP \neq 0$ holds by Proposition \ref{prop-up-down-finite}. Thus
\[\ind(\cT\cap\cP)=\ind(\cT^0\cap\cP) \supset \ind(\cT^1\cap\cP) \supset\cdots\]
is a descending chain of non-empty sets.
If $\ind(\cT\cap \cP)$ is a finite set, then there exists $P\in\bigcap_{i\ge0}\ind(\cT^i\cap\cP)$.
However this is a contradiction since $\gen P$ is upper finite by Proposition \ref{prop-up-down-finite}(a).
Therefore $\ind(\cT\cap \cP)$ is an infinite set.
\end{proof}

We say that an indecomposable preprojective (respectively, preinjective) $kQ$-module $X$ is \emph{$\tau^-$-sincere} (respectively, \emph{$\tau$-sincere}) if for any integer $i\geq 0$, $\tau^{-i}X$ (respectively, $\tau^iX$) is sincere.
%%%%%%%%%%%%%%%
\begin{lem}\label{lem-pp-hom-non-zero}
Let $P\in\ind\cP$, $I\in\ind\cI$, and $X\in\mod kQ$ such that $\tau X \simeq X$.
\begin{enumerate}[{\rm (a)}]
	\item If $P$ is $\tau^-$-sincere (respectively, $I$ is $\tau$-sincere), then $\Hom_{kQ}(P, I)\neq 0$ holds.
	\item We have that $\Hom_{kQ}(P, X)\neq 0$ and $\Hom_{kQ}(X, I)\neq 0$.
\end{enumerate}
\end{lem}
%%%%%
\begin{proof}
(a)
There is an integer $\ell\geq 0$ such that $\tau^{-\ell}I$ is an indecomposable injective module.
Then, since $\tau^{-\ell}P$ is sincere, we have $\Hom_{kQ}(P, I)=\Hom_{kQ}(\tau^{-\ell}P, \tau^{-\ell}I)\neq 0$.

(b)
We only show the former statement since the latter one is the dual.
Since $\tau X \simeq X$ and $Q$ is extended Dynkin, $X$ is sincere.
Take an integer $m\geq 0$ such that $\tau^mP$ is an indecomposable projective module.
Then $\Hom_{kQ}(P, X)=\Hom_{kQ}(\tau^m P,\tau^m X)=\Hom_{kQ}(\tau^m P,X)\neq 0$ as desired.
\end{proof}
%%%%%%%%%%%%%%%
\begin{prop}\label{prop-genP-H-I}
Let $P\in\cP$ be indecomposable and $\tau^-$-sincere. Then $\cH \vee \cI \subseteq\gen P$ holds.
\end{prop}
%%%%%
\begin{proof}
It suffices to show that each indecomposable object $X\in\cI$ and each simple object $X\in\cH$ belong to $\gen P$.
Let $f : P' \to X$ be a right $(\add P)$-approximation of $X$, and $C:=\Cok f$.
To prove $X\in\gen P$, it suffices to show $C=0$. 

We claim $C\in\cI$. If $X\in\cI$, then this is clear. If $X\in\cH$ is simple, then $\tau X\simeq X$ holds and hence $f\neq0$ by Lemma \ref{lem-pp-hom-non-zero}(b). Since $C\notin\cR$, we have $C\in\cI$.

Assume $C\neq 0$. By Lemma \ref{lem-pp-hom-non-zero}(a), there is a non-zero morphism $g : P\to C$.
Since $\Ext_{kQ}^1(P, \gen P)=0$, this $g$ factors through $\pi : X \to C$.
Since $f$ is a right $(\add P)$-approximation, $g$ factors through $\pi \circ f =0$, a contradiction.
Thus $C=0$ holds, as desired.
\end{proof}

%\begin{prop}\label{prop-genP-H-I}
%\new{The following assertions hold. 
%\begin{enumerate}[\rm(a)]
%\item If $P\in\cP$ is indecomposable and $\tau^-$-sincere, then $\cH \vee \cI \subseteq\gen P$ holds.
%\item Let $\cT\in\tors kQ$. If $\ind(\cT \cap \cP)$ is an infinite set, then $\cH \vee \cI \subseteq\cT$ holds.
%\end{enumerate}
%\end{prop}

\begin{proof}[Proof of Theorem \ref{upper not lower}]
We only prove (a) since (b) is its dual. The first assertion is Proposition \ref{prop-up-down-finite}.

(i)$\Rightarrow$(ii) This follows from Lemma \ref{lem-tors-pp-inf}(b).

(ii)$\Rightarrow$(iii) Since all but finitely many indecomposable preprojective $kQ$-modules are sincere by \cite[Proposition IX.5.6]{ASS}, $\cT$ contains a $\tau^-$-sincere indecomposable preprojective module $P$. By Proposition \ref{prop-genP-H-I}, we have $\cH\vee \cI\subseteq\Fac P\subseteq\cT$. 

(iii)$\Rightarrow$(i) It suffices to show that $\cT$ is lower infinite. This is clear since $\cT$ contains $\cT_n:=\add\{\tau^i(D(kQ))\mid 0\le i\le n\}\in\tors kQ$ for each $n\ge0$.
\end{proof}

%%%%%%%%%%%%%%%

%%%%%%%%%%%%%%%
%%%%%%%%%%%%%%%
\section{Preliminaries on Noetherian algebras}\label{section-pre-Noe}
\subsection{Basic results}\label{subsection-basic-result}
In this subsection, we prove basic properties of Noetherian algebras.
%%%%%%%%%%%%%%%
\begin{lem}\cite[(23.13) Proposition]{Curtis-Reiner}\label{lem-local-dense}
Let $R$ be a commutative Noetherian ring, $\p\in\Spec R$, and $\L$ a Noetherian $R$-algebra.
Then the functor $(-)_\p:\mod\L\to\mod\L_\p$ is dense.
\end{lem}
%%%%%%%%%%%%%%%
Let $R$ be a commutative ring. For an ideal $I$ of $R$, let $V(I):=\{\p\in\Spec R\mid\p\supseteq I\}$. For an element $r\in R$, we denote by $R_r$ a localization of $R$ by a multiplicative set $\{r^i \mid i \geq 0\}$.
We start with the following standard observation.
%%%%%%%%%%%%%%%
\begin{lem}\label{lem-ractional-NZD}
Let $R$ be a commutative Noetherian ring and $\L$ a Noetherian $R$-algebra.
For $X, Y\in\sD^{\rm b}(\mod \L)$, assume that there exists $\p\in\Spec R$ such that $X_\p \simeq Y_\p$.
Then there exists $r\in R\setminus \p$ such that $X_r \simeq Y_r$.
In particular $U=\Spec R \setminus V((r))$ is a Zariski open neighborhood of $\p$ such that $X_\q \simeq Y_\q$ for any $\q \in U$.
\end{lem}
%%%%%
\begin{proof}
Let $f\in\Hom_{\sD^{\rm b}(\mod \L_\p)}(X_\p,Y_\p)$ be an isomorphism and $g:=f^{-1}$. Since $\Hom_{\sD^{\rm b}(\mod\L_\p)}(X_\p, Y_\p) \simeq R_\p \otimes_{R} \Hom_{\sD^{\rm b}(\mod \L)}(X, Y)$ holds, there exists $r\in R\setminus \p$ such that $rf$ and $rg$ are morphisms in $\sD^{\rm b}(\mod \L)$.
Then $f$ and $g$ are morphisms in $\sD^{\rm b}(\mod \L_r)$ such that there exists $r'\in R\setminus\p$ such that $r'gf=r'1_X$ and $r'fg=r'1_Y$. Then $f:X_{rr'}\to Y_{rr'}$ and $g:Y_{rr'}\to X_{rr'}$ are inverses of each other.
\end{proof}

%We denote by $\siltm\L$ the set of additive equivalent classes of silting $\L$-modules.
%%%%%%%%%%%%%%%
%%%%%%%%%%%%%%%

\subsection{Methods of reverse filtrations}
We introduce the following notation which plays an important role to prove some properties of the map ${\rm r}_{\p\q}$.

%We end this section with the following easy observations, which will be use later.

\begin{dfn}\label{define reverse}
Let $A$ be a finite dimensional algebra over a field $k$, and $X,Y\in\mod A$.
We say that $X$ and $Y$ admit \emph{reverse filtrations} of each other if they have $A$-submodules
\begin{align}\label{reverse}
0=X^{\ell} \subseteq \dots \subseteq X^1 \subseteq X^0 =X\ \mbox{ and }\ 0=Y^0 \subseteq Y^1 \subseteq \dots \subseteq Y^{\ell} =Y
\end{align}
such that $Z^i:=X^{i-1}/X^i \simeq Y^i/Y^{i-1}$ holds as $\L$-modules for $1 \leq i \leq \ell$.
\end{dfn}

The observation below shows that two modules admitting reverse filtrations appear naturally in our study.
%We give a sufficient condition such that two modules admit reverse filtrations each other.
It plays a key role in the proof of Theorem \ref{thm-compatible-extDynkin}.
\begin{prop}\label{prop-filt-eq}
Let $R$ be a discrete valuation ring with the maximal ideal $\m=R\pi$, $k=R/\m$ and $\L$ a Noetherian $R$-algebra.
Let $X, Y\in\mod\L$ such that $X, Y$ are free of finite rank over $R$ and $X_0\simeq Y_0$.
Then $k\otimes_{R}X$ and $k\otimes_{R}Y$ admit reverse filtrations.
%finite filtrations 
%\[0=X^m \subseteq \dots \subseteq X^{1}\subseteq  X^{0} =k\otimes_{R}X\ \mbox{ and }\ 0=Y^{0} \subseteq Y^{1} \subseteq \dots \subseteq Y^m =k\otimes_{R}Y\]
%such that $X^{i-1}/X^{i} \simeq Y^{i}/Y^{i-1}$ holds for $1 \leq i \leq m$.
\end{prop}
%%%%%
\begin{proof}
Let $K$ be a quotient field of $R$.
Fix $R$-bases of $\cX$ of $X$ and $\cY$ of $Y$. Then the isomorphism $X_0\simeq Y_0$ is given by a matrix $A\in\GL_n(K)$.
Taking a Smith normal form, there are $P, Q\in \GL_n(R)$ and $D=\diag(D_{-m},\ldots,D_0,D_1,\ldots,D_m)$ with $D_i=\pi^iI_{n_i}$ such that $A=PDQ$.
We may assume $A=D$. Take decompositions $\cX=\bigsqcup_{1\le i\le m}\cX_i$ and $\cY=\bigsqcup_{1\le i\le m}\cY_i$, where $\cX_i$ and $\cY_i$ correspond to the block $D_i$.

For $a\in\L$, let $M^a \in \M_n(R)$ (respectively, $N^a$) be a representation matrix of an action of $a$ on $X$ (respectively, $Y$) with respect to $\cX$ and $\cY$.
We write it as a block matrix $M^a=(M^a_{ij})_{1\le i,j\le m}$ (respectively, $N^a=(N^a_{ij})_{1\le i,j\le m}$) with respect to the decomposition $\cX=\bigsqcup_{1\le i\le m}\cX_i$ (respectively, $\cY=\bigsqcup_{1\le i\le m}\cY_i$).
Since $DM^a = N^aD$ holds, $\pi^iM^a_{ij}=\pi^{j}N^a_{ij}$ holds for each $i,j$. Thus the following properties hold.
\begin{enumerate}[\rm(i)]
\item If $i<j$ (respectively, $i>j$), then the entries of $M^a_{ij}$ (respectively, $N^a_{ij}$) belong to $\m$.
\item $M^a_{ii}=N^a_{ii}$ for each $i$.
\end{enumerate}
We denote by $\overline{(-)}:R\to k$ the natural projection. Then the action of $1\otimes a \in k \otimes_R \L$ on $k\otimes_R X$ (respectively, on $k\otimes_R Y$) with respect to the $k$-bases $1\otimes\cX$ and $1\otimes\cY$ is given by the matrix $\overline{M^a}=(\overline{M^a_{ij}})_{i,j}$ (respectively, $\overline{N^a}=(\overline{N^a_{ij}})_{i,j}$), which is block lower (respectively, upper) triangular by (i).
For each $i$, let $X^i$ (respectively, $Y^i$) be the subspace of $k\otimes_RX$ with basis $\bigsqcup_{i<j\le m}\cX_j$ (respectively, $k\otimes_RY$ with basis $\bigsqcup_{-m\le j\le i}\cY_j$) gives a $k\otimes_R\L$-submodule.
Then the filtrations $0=X^{m} \subseteq X^{m-1} \subseteq \dots \subseteq X^{-m-1} =k\otimes_{R}X$ and $0=Y^{-m-1}\subseteq \dots \subseteq Y^{m-1} \subseteq  Y^m=k\otimes_{R}Y$ satisfy $X^{i-1}/X^{i} \simeq Y^i/Y^{i-1}$ for each $i$ by (ii). Shifting the indices, we obtain the desired filtrations.
\end{proof}

We prepare the following technical observations.
For a finite dimensional algebra $A$ and $X\in\mod\L$, we denote by $\sT(X)$ the smallest torsion class of $\mod A$ containing $X$.

\begin{lem}\label{reverse lemma}
Let $Q$ be an extended Dynkin quiver. Assume that $X,Y\in\mod kQ$ admit reverse filtrations of each other as \eqref{reverse}. Then the following assertions hold.
\begin{enumerate}[\rm(a)]
\item If $X^{\rm pi}=0$ and there exists $1\le i\le\ell$ such that $Z^i\notin\cR$ holds in \eqref{reverse}, then $Y^{\rm pp}\neq0$. %$Y\notin\cR\vee\cI$. 
\item If $X^{\rm pi}=0$ and $X^{\rm pp}\neq0$, then $Y^{\rm pp}\neq0$.
\item If $X\in\cH$, then either $Y^{\rm pp}\neq0$ or {\rm(}$Y\in\cH$ and $\sT(X)=\sT(Y)${\rm)} holds.
\end{enumerate}
\end{lem}

\begin{proof}
(a) Let $i$ be the maximal integer such that $Z^i\notin\cR$.
Consider an exact sequence
\[
\xymatrix{0\ar[r]&
{X^{i} =
\left(
\begin{smallmatrix}
Z^{i+1} \\
Z^{i+2} \\
\vdots \\
Z^{\ell}
\end{smallmatrix}
\right)}\ar[r]&
{X=
\left(
\begin{smallmatrix}
Z^1 \\
Z^2 \\
\vdots \\
Z^{\ell}
\end{smallmatrix}
\right)}
\ar[r]&
{X/X^i =
\left(
\begin{smallmatrix}
Z^{1} \\
Z^{2} \\
\vdots \\
Z^{i}
\end{smallmatrix}
\right)}\ar[r]&0
}
\]
Since $X\in\cP\vee\cR$ and $X^i\in\cR$ holds by our choice of $i$, we have $X/X^i\in\cP\vee\cR$.
Since $Z^i$ is a submodule of $X/X^i$, we have $Z^i\in(\cP\vee\cR)\setminus\cR$.
Since $Y/Y^{i-1}$ has a filtration
%$Y/Y^{i-1}=\left(
%\begin{smallmatrix}
%Z^{\ell} \\
%Z^{\ell-1} \\
%\vdots \\
%Z^{i}
%\end{smallmatrix}
%\right)$
by $Z^j\in\cR$ with $i+1\le j\le\ell$ and $Z^i\notin\cR$, we have $Y/Y^{i-1}\in(\cP\vee\cR)\setminus\cR$. Since $Y$ has $Y/Y^{i-1}$ as a factor module, we have $Y^{\rm pp}\neq0$.

(b) Since $X\notin\cR$, there exists $1\le i\le\ell$ such that $Z^i\notin\cR$. Thus the assertion follows from (a).

(c) If there exists $1\le i\le\ell$ such that $Z^i\notin\cR$, then (a) implies $Y^{\rm pp}\neq0$.
Thus we can assume that $Z^i\in\cR$ holds for each $i$.
Refining the filtrations, we may assume that all $Z^i$ are indecomposable.
Then all $Z^i$ belongs to $\cH$ by our assumption, and hence $Y\in\cH$ holds.
Since there are no extensions between two modules which belong to different tubes, we have $\sT(X) = \sT(Y)$.
\end{proof}

Later we will use the following result.

\begin{lem}\label{lem-XPHYT}
Let $(R,\m,k)$ be a discrete valuation ring with quotient field $K$, and $Q$ an extended Dynkin quiver.
Let $X,Y\in\mod RQ$ such that $X,Y$ are free of finite rank over $R$ and $X_0\simeq Y_0$.
Let $\cT\in\ntors kQ$.
If $k\otimes_RX\in(\cP^k\vee\cH^k)\setminus\cT$, then $k\otimes_RY\notin\cT$.
\end{lem}

\begin{proof}
Proposition \ref{prop-filt-eq}, $M:=k\otimes_RX$ and $N:=k\otimes_RY$ admit reverse filtrations.
Since $M\in\cP^k\vee\cH^k$, we have $M^{\rm pi}=0$.
First, assume $M\notin\cH^k$. Then $M^{\rm pp}\neq0$ holds, and Lemma \ref{reverse lemma}(b) implies $N^{\rm pp}\neq0$. Since $\cT\in\ntors RQ=[\cI,\cR\vee\cI]$, we have $N\notin\cT$, as desired.
Next, assume $M\in\cH^k$. Then Lemma \ref{reverse lemma}(c) implies $\sT(N)=\sT(M)\ {\not\subseteq}\ \cT$, and hence $N\notin\cT$ holds, as desired.
\end{proof}

%%%%%%%%%%%%%%%
%%%%%%%%%%%%%%%
\section{Preliminaries on path algebras of acyclic quivers over Noetherian rings}\label{section-path-Noe}

\subsection{Real Schur roots and exceptional modules}
%\new{
Let $Q$ a finite acyclic quiver.
We denote by $\Phi^{\rm rS}$ the set of real Schur roots associated to $Q$.
For a field $k$, $X\in\mod kQ$ is \emph{exceptional} if $\Ext_{kQ}^1(X, X)=0$ and $\End_{kQ}(X)\simeq k$.
We denote by $\excep kQ$ the set of the isomorphism classes of exceptional $kQ$-modules.
By Kac's theorem \cite{Kac}, there exists a bijection 
\begin{align*}
\Dim : \excep kQ \simeq \Phi^{\rm rS}.
\end{align*}
Crawley-Boevey showed that this is a specialization of a certain bijection for $\bZ Q$-modules \cite{CB}.
%%%%%%%%%%%%%%%
%Let $R$ be a commutative ring and let $X\in\mod RQ$.
%If $e_iX$ is free $\bZ$-module of finite rank for each $i\in Q_0$, then 
A $\bZ Q$-module $X$ is \emph{exceptional} if $\Ext_{\bZ Q}^1(X, X)=0$ and $\End_{\bZ Q}(X)\simeq \bZ$.
We denote by $\excep \bZ Q$ the set of the isomorphism classes of exceptional $\bZ Q$-modules.
Then each $X\in\excep \bZ Q$ is free as a $\bZ$-module \cite[Theorem 2(1)]{CB}, and there exists a bijection \cite[Theorem 1]{CB}
%We say that $X\in\mod \bZ Q$ is a \emph{lattice} if $X$ is a free as a $\bZ$-module.
\[
\urank : \excep \bZ Q \simeq \Phi^{\rm rS}, \quad X \mapsto \urank X:=(\rank e_i X)_i.
\]
We denote the inverse map by
\[M_{-} : \Phi^{\rm rS}\simeq\excep \bZ Q,\quad \alpha \mapsto M_{\alpha}.\]
%A $\bZ Q$-lattice $X$ is \emph{rigid} if $\Ext_{\bZ Q}^1(X, X)=0$, and is \emph{exceptional} if it is rigid and a natural morphism $\bZ\to \End_{\bZ Q}(X)$ is an isomorphism. Let $\excep \bZ Q$ be the set of isomorphism classes of exceptional $RQ$-lattices. Then there exists a bijection \[\Phi^{\rm rS}\simeq\excep \bZ Q,\quad \alpha \mapsto M_{\alpha}\] such that $\urank M_\alpha=\alpha$ \cite[Theorem 1]{CB}.
Using this, we introduce a family of modules which plays an important role in this paper (\cite[Section 3]{IK2}, \cite{CB3}):
For a commutative ring $R$ and $\alpha \in\Phi^{\rm rS}$, we define an $RQ$-module by % have a map 
\begin{equation}\label{prop-cluster-stilt}
M_{\alpha}^R:=R\otimes_{\bZ}M_{\alpha}.
%\Phi^{\rm rS}  \to \excep RQ, \quad \alpha \mapsto M_{\alpha}^R:=R\otimes_{\bZ}M_{\alpha},
%\item $\sP_{\rm cl}(\Phi_{\ge-1}^{\rm rS})\to \stilt RQ, \quad \cS \mapsto M_{\cS}^R$.
\end{equation}
By construction, $M^R_{\alpha}$ is free of finite rank as an $R$-module, and satisfies $\Ext^1_{RQ}(M^R_{\alpha}, M^R_{\alpha})=0$, $\End_{RQ}(M^R_{\alpha})\simeq R$ and $\projdim_{RQ} M^R_{\alpha} \leq 1$.
In particular, $M^R_{\alpha}$ is a partial tilting $RQ$-module.
Moreover, for each field $k$, $M^k_- : \Phi^{\rm rS} \simeq \excep kQ$ is a bijection whose inverse is $\Dim$.

%Let $\Clus(Q)$ be the partially ordered set of clusters of $Q$ (\cite[Definition 3.1]{IK2}). By \cite[Theorem 3.4]{IK2}, there exist isomorphisms of posets \begin{align}\label{iso:MPH} M_{-}^R:\Clus(Q) \xrightarrow[\sim]{\ P_{-}^R\ } \twosilt RQ\xrightarrow[\sim]{\ H^0\ }\stilt RQ. \end{align}

%which is bijective if $R$ is $\bZ$ or a field \cite{IK2}. \comment{Add a reference}
For an extended Dynkin quiver $Q$, we have the following well-known observation.
%%%%%%%%%%%%%%%
\begin{prop}\label{prop-tube-k-K}
Let $Q$ be an extended Dynkin quiver. Then there exists a decomposition
\[\Phi^{\rm rS}=\Phi_{\rm pp}^{\rm rS}\sqcup\Phi_{\rm rg}^{\rm rS}\sqcup\Phi_{\rm pi}^{\rm rS}\]
such that, for an arbitrary field $k$, the bijection $M^k_{-}:\Phi^{\rm rS} \simeq \excep kQ$, $\alpha\mapsto M_\alpha^k$ is a gluing of the following bijections.
\begin{enumerate}[{\rm (i)}]
	\item A bijection $\Phi_{\rm pp}^{\rm rS}\simeq\ind \cP^k$.
	\item A bijection $\Phi_{\rm rg}^{\rm rS}\simeq\excep kQ \cap \cR^k$.
	\item A bijection $\Phi_{\rm pi}^{\rm rS}\simeq\ind \cI^k$.
\end{enumerate}
Moreover $\Phi_{\rm rg}^{\rm rS}$ is a finite set.
\end{prop}
%%%%%
%Although this is folklore, we give a complete proof for convenience of the reader.

\begin{proof}
Fix a field $k$, and consider the bijection $\Phi^{\rm rS}\simeq\excep kQ$ given in Proposition \ref{prop-cluster-stilt}.
Since each module in $\ind\cP^k\sqcup\ind\cI^k$ is exceptional \,\cite[IV.\,Corollary 2.15]{ASS}, $\excep kQ$ is a disjoint union $\ind\cP^k\sqcup(\excep kQ \cap \cR^k)\sqcup\ind\cI^k$.
Thus we obtain the corresponding decomposition $\Phi^{\rm rS}=\Phi_{\rm pp}^{\rm rS}\sqcup\Phi_{\rm rg}^{\rm rS}\sqcup\Phi_{\rm pi}^{\rm rS}$.
This decomposition does not depend on the choice of $k$ since objects in $\excep kQ$ are determined by the dimension vectors.

We prove the last statement.
An exceptional regular module belongs to a tube of $\cR_{\geq 2}$ and its regular length is at most the number of simple modules of the tube.
Thus $\Phi_{\rm rg}^{\rm rS}$ is a finite set.
\end{proof}
%%%%%%%%%%%%%%%
The following easy observation will be used later.
%%%%%%%%%%%%%%%
\begin{lem}\label{factor of M}
Let $Q$ be an extended Dynkin quiver and $\alpha\in\Phi_{\rm rg}^{\rm rS}\sqcup\Phi_{\rm pi}^{\rm rS}$. 
Then each factor $kQ$-module $N$ of $M_\alpha^k$ is a direct sum of $kQ$-modules of the form $M_\beta^k$ with $\beta\in\Phi_{\rm rg}^{\rm rS}\sqcup\Phi_{\rm pi}^{\rm rS}$.
\end{lem}
%%%%%
\begin{proof}
If $\alpha\in\Phi_{\rm pi}^{\rm rS}$, then $M_\alpha^k\in\cI^k$ and its factor module also belongs to $\cI^k$. Thus the assertion holds.
Assume $\alpha\in\Phi_{\rm rg}^{\rm rS}$. Then $N=N^{\rm rg}\oplus N^{\rm pi}$ holds.
Since $N^{\rm rg}$ is a factor module of $M_\alpha^k$ which is uniserial as an object in $\cR^k$, it follows that
%and $M_\alpha^k$ and $N^{\rm rg}$ belong to the same tube of rank at least $2$.
%$M_\alpha^k\in\cR_x$ for some $x\in|\bP^1_k|_{\ge2}$. 
$N^{\rm rg}$ is also an exceptional $kQ$-module. Thus the assertion holds.
\end{proof}

%%%%%%%%%%%%%%%
%%%%%%%%%%%%%%%
\subsection{Derived Auslander-Reiten translations}\label{subsection-DAR-trans}
%%%%%%%%%%%%%%%
For a commutative ring $R$ and a finite acyclic quiver $Q$, the \emph{derived Auslander-Reiten translation} is an endofunctor of $\sD(\Mod RQ)$ defined by
\[
{\boldsymbol{\tau}}_{RQ}(-):=\Hom_{R}(RQ, R)[-1]\otimes_{RQ}^{\bf L}(-).
\]
%%%%%%%%%%%%%%%
\begin{prop}\label{prop-global-tau}
Let $R\to S$ a ring homomorphism of commutative rings, and $Q$ a finite acyclic quiver.
\begin{enumerate}[{\rm (a)}]
	\item We have an isomorphism $S\otimes_R\Hom_R(RQ, R) \simeq \Hom_{S}(S Q, S)$ as $S Q$-$RQ$-bimodules.
	\item We have the following isomorphism of triangle functors $\sD(\Mod RQ)\to\sD(\Mod S Q)$:
	\[
	S\otimes_R^{\bf L} {\boldsymbol{\tau}}_{RQ}(-) \simeq {\boldsymbol{\tau}}_{S Q}(S\otimes_R^{\bf L} (-)).
	\]
\end{enumerate}
\end{prop}
%%%%%
\begin{proof}
(a)
We have $S\otimes_R\Hom_R(RQ, R) \simeq \Hom_R(RQ, S) \simeq \Hom_S(S Q,S)$.

(b)
We have
\begin{align*}
S\otimes_R^{\bf L} {\boldsymbol{\tau}}_{RQ}(-) & \simeq S\otimes_R^{\bf L}\Hom_{R}(RQ, R)[-1]\otimes_{RQ}^{\bf L}(-)\\
& \stackrel{\rm(a)}{\simeq} \Hom_S(S Q,S)[-1]\otimes_{RQ}^{\bf L}(-)\\
& \simeq \Hom_S(S Q,S)[-1]\otimes_{S Q}^{\bf L}(S\otimes_R^{\bf L} RQ)\otimes_{RQ}^{\bf L}(-)\\
& \simeq \Hom_S(S Q,S)[-1]\otimes_{S Q}^{\bf L}\left(S\otimes_{R}^{\bf L}(-) \right)\\
& \simeq {\boldsymbol{\tau}}_{S Q}(S\otimes_R^{\bf L} (-)).\qedhere
\end{align*}
\end{proof}
%%%%%%%%%%%%%%%

%%%%%%%%%%%%%%%
Using the derived Auslander-Reiten translation, we deduce a useful observation on the usual Auslander-Reiten translation $\tau$.
%%%%%%%%%%%%%%%
\begin{lem}\label{lem-tau-cyclic-lift}
Let $R$ be a commutative Noetherian ring, and $Q$ a finite acyclic quiver.
For $\q\in\Spec R$ and an integer $\ell\geq 1$, the following statements hold.
\begin{enumerate}[{\rm (a)}]
	\item Let $X, Y\in\sD^{\rm b}(\mod RQ)$. If ${\boldsymbol{\tau}}^{\ell}_{R_\q Q}(X_\q) \simeq Y_\q$ in $\sD^{\rm b}(\mod R_\q Q)$, then there exists a Zariski open neighborhood $\q\in U\subset\Spec R$ such that ${\boldsymbol{\tau}}^{\ell}_{\kappa(\p)Q}(\kappa(\p)\otimes_R^{\bf L} X) \simeq \kappa(\p)\otimes_R^{\bf L} Y$ for any $\p \in U$.
	\item Assume that $\q$ is a minimal prime ideal of $R$ such that $R_\q$ is a field. Let $X, Y\in\mod RQ$. If $\tau^{\ell}(X_\q) \simeq Y_\q$ in $\mod R_\q Q$, then there exists a Zariski open neighborhood $\q\in U\subset\Spec R$ such that $\tau^{\ell}(\kappa(\p)\otimes_R X) \simeq \kappa(\p)\otimes_R Y$ for any $\p \in U$.
\end{enumerate}
\end{lem}
%%%%%
\begin{proof}
(a)
Let $Z:={\boldsymbol{\tau}}^{\ell}_{RQ}(X)$.
Applying Proposition \ref{prop-global-tau}(b) to the canonical map $R\to R_\q$, we have $Z_\q \simeq {\boldsymbol{\tau}}^{\ell}_{R_\q Q}(X_\q) \simeq Y_\q$.
By Lemma \ref{lem-ractional-NZD} there exists an open neighborhood $U$ of $\q$ such that $Z_\p \simeq Y_\p$ holds for any $\p\in U$.
In particular $\kappa(\p)\otimes_R^{\bf L} Z \simeq \kappa(\p)\otimes_{R_\p}^{\bf L} Z_\p \simeq  \kappa(\p)\otimes_{R_\p}^{\bf L} Y_\p \simeq \kappa(\p)\otimes_R^{\bf L} Y$ holds for any $\p\in U$.
Applying Proposition \ref{prop-global-tau}(b) to the canonical map $R\to\kappa(\p)$, we obtain ${\boldsymbol{\tau}}^{\ell}_{\kappa(\p)Q}(\kappa(\p)\otimes_R^{\bf L} X) \simeq \kappa(\p)\otimes_R^{\bf L} Z \simeq \kappa(\p)\otimes_R^{\bf L} Y$ for any $\p\in U$.

(b)
Let $Y^0:=Y$, and
for $i<0$, we take $Y^i\in\mod RQ$ such that $(Y^i)_\q \simeq H^{i} :=H^{i}({\boldsymbol{\tau}}^{\ell}_{R_\q Q}(X_\q))$ by Lemma \ref{lem-local-dense}(a). Let $Z:=\bigoplus_{i\le0} Y^i[-i]$.
Then $Z_\q= \bigoplus_{i\le 0}H^{i}[-i]={\boldsymbol{\tau}}^{\ell}_{R_\q Q}(X_\q)$ holds since $R_\q Q$ is a hereditary algebra over a field $R_\q$.
Applying (a) and Lemma \ref{lem-ractional-NZD} for $\Lambda:=R$, there exists an open neighborhood $U \subset \Spec R$ of $\q$ such that the following conditions are satisfied for each $\p\in U$.
\begin{enumerate}[\rm(i)]
\item $\kappa(\p)\otimes_R^{\bf L} Z \simeq {\boldsymbol{\tau}}^{\ell}_{\kappa(\p)Q}(\kappa(\p)\otimes_R^{\bf L} X)$.
\item $X_\p\in\proj R_\p$.
\end{enumerate}
For each $\p\in U$, we have $\kappa(\p)\otimes_R^{\bf L} X=\kappa(\p)\otimes_{R_\p}^{\bf L} X_\p\stackrel{{\rm (ii)}}{=}\kappa(\p)\otimes_{R_\p} X_\p=\kappa(\p)\otimes_RX$.
Taking $H^0$, we have
\begin{align*}
\kappa(\p)\otimes_R Y=H^0(\kappa(\p)\otimes_R^{\bf L} Y)&= H^0(\kappa(\p)\otimes_R^{\bf L} Z) \stackrel{{\rm (i)}}{\simeq} H^0({\boldsymbol{\tau}}^{\ell}_{\kappa(\p)Q}(\kappa(\p)\otimes_R^{\bf L} X))=\tau^{\ell}(\kappa(\p)\otimes_R X).\qedhere
\end{align*}
\end{proof}
%%%%%%%%%%%%%%%
Note that if $R$ satisfies $(R_1)$, then $R_\q$ is a field for each minimal prime ideal $\q$ of $R$.
The following lemma is used in the proof of Theorem \ref{thm-compatible-extDynkin}.

%%%%%%%%%%%%%%%
\begin{lem}\label{lem-M-0}
Let $R$ be a commutative Noetherian ring satisfying \emph{($R_1$)}, and $\q$ a minimal prime ideal of $R$.
Let $Q$ be an extended Dynkin quiver, and $X, Y\in\mod RQ$.
\begin{enumerate}[{\rm (a)}]
\item If $\Ext^i_{R_\q Q}(X_\q, Y_\q)=0$ for $i=0,1$ and $Y\in\proj R$, then there exists a Zariski open neighborhood $\q\in U\subset\Spec R$ such that $\Hom_{\kappa(\p)Q}(\kappa(\p)\otimes_R X, \kappa(\p)\otimes_R Y)=0$ for any $\p\in U$ of height one.
\item If $X_\q\in\cH^{R_\q}$, then there exists a Zariski open neighborhood $\q\in U\subset\Spec R$ such that $\kappa(\p)\otimes_R X \in\cH^{\kappa(\p)}$ for any $\p\in U$ of height one.
\end{enumerate}
\end{lem}
%%%%%
\begin{proof}
(a)
We prove that $U:=\Spec R \setminus \bigcup_{i=0,1}\supp(\Ext^i_{RQ}(X, Y))$ satisfies the desired condition. Clearly $\q\in U$ holds by our assumption.
For each $\p\in U$ of height one, we  can take $\pi\in R_\p$ such that $\p R_\p=\pi R_\p$ thanks to ($R_1$) condition. Since $Y_\p\in\proj R_\p$ by our assumption, we have an exact sequence
\[
0 \to Y_\p \xto{\pi} Y_\p\to \kappa(\p)\otimes_R Y \to 0.
\]
Applying $\Hom_{R_\p Q}(X_\p, -)$, we obtain an exact sequence
\[
0=\Hom_{R_\p Q}(X_\p, Y_\p) \xto{f} \Hom_{R_\p Q}(X_\p, \kappa(\p)\otimes_R Y)\to \Ext^1_{R_\p Q}(X_\p, Y_\p)=0.
\]
Thus $0=\Hom_{R_\p Q}(X_\p, \kappa(\p)\otimes_R Y)=\Hom_{\kappa(\p) Q}(\kappa(\p)\otimes_RX, \kappa(\p)\otimes_R Y)$ holds as desired.

(b)
Since $\tau(X_\q)\simeq X_\q$ and $R_\q$ is a field by ($R_1$), Lemma \ref{lem-tau-cyclic-lift}(b) implies that there is an open neighborhood $\q\in U\subset\Spec R$ such that $\tau(\kappa(\p)\otimes_R X) \simeq \kappa(\p)\otimes_R X$ holds for any $\p\in U$.
Thus we have
\begin{equation}\label{k(p)X in R}
\kappa(\p)\otimes_R X\in\cR^{\kappa(\p)}\ \mbox{ for any }\ \p\in U.
\end{equation}
%\old{Recall the abelian category $\cR_{\geq2}^k$ has only finitely many isomorphism classes of simple objects. We denote their direct sum by $S$. Let $Y:=R\otimes_kS$.}
%\old{For a ring $R$, let $M_\alpha^R:=R\otimes_{\bZ}M_\alpha$. Let $\Phi^{\rm rS}_{+,\cR}$ be the subset consisting of all $\alpha$ such that $M_\alpha^k\in\cR^k$.}
Let $Y:=\bigoplus_{\alpha\in\Phi_{\rm rg}^{\rm rS}}M_\alpha^R\in\mod RQ$.
Since $X_\q\in\cH^{R_\q}$ and $Y_\q\in\cR_{\geq 2}^{R_\q}$, we have $\Ext^i_{R_\q Q}(X_\q, Y_\q)=0$ for $i=0,1$. By (a), there is an open neighborhood $\q\in U'\subset\Spec R$ such that
\begin{equation}\label{k(p)X,k(p)Y}
\mbox{$\Hom_{\kappa(\p)Q}(\kappa(\p)\otimes_R X, \kappa(\p)\otimes_R Y)=0$ holds for any $\p \in U'$ of height one}.
\end{equation} 
To prove that $U \cap U'$ satisfies the desired condition, fix $\p\in U\cap U'$ of height one. By Proposition \ref{prop-tube-k-K}(ii), $\kappa(\p)\otimes_RY=\bigoplus_{\alpha\in\Phi_{\rm rg}^{\rm rS}}M_\alpha^{\kappa(\p)}$ has all simple objects in $\cR_{\geq 2}^{\kappa(\p)}$ as a direct summand. Thus \eqref{k(p)X,k(p)Y} implies $\Hom_{\kappa(\p)Q}(\kappa(\p)\otimes_R X, \cR_{\geq 2}^{\kappa(\p)})=0$.
Since $\cR^{\kappa(\p)}=\cR_{\geq2}^{\kappa(\p)}\vee\cH^{\kappa(\p)}$ holds, \eqref{k(p)X in R} implies $\kappa(\p)\otimes_R X \in \cH^{\kappa(\p)}$ as desired.
\end{proof}
%%%%%%%%%%%%%%%
%%%%%%%%%%%%%%%
\section{Torsion classes of $RQ$}
%%%%%%%%%%%%%%%

Let $R$ be a commutative Noetherian ring and $\L$ a Noetherian $R$-algebra. We denote by $\sK^{\rm b}(\proj \L)$ the homotopy category of the category $\proj \L$ of finitely generated projective $\L$-modules. We say that $P\in \sK^{\rm b}(\proj \L)$ is \emph{presilting} if $\Hom(P, P[i])=0$ holds for any $i>0$, and $P$ is \emph{silting} if it is presilting and the smallest thick subcategory of $\sK^{\rm b}(\proj \L)$ containing $P$ is $\sK^{\rm b}(\proj \L)$. We say that a complex in $\sK^{\rm b}(\proj \L)$ is \emph{two-term} if it is concentrated in degree $-1$ and $0$. A \emph{two-term silting complex} is a two-term complex which is silting. We denote by $\twosilt\L$ the set of additive equivalent classes of two-term silting complexes, where $X$ and $Y$ are additive equivalent if $\add X = \add Y$ holds.

%A \emph{silting $\L$-module} is a module $M\in\mod\L$ such that there exists a two-term silting complex $P$ with $\add M = \add H^0(P)$. We denote by $\siltm\L$ the set of additive equivalent classes of silting $\L$-modules.
%By taking the 0-th cohomology, we have an isomorphism $H^0 : \twosilt \L \to \siltm \L$ \cite[Proposition A.6]{Iyama-Kimura}.}

\subsection{Properties of ${\rm r_{\p\q}}$}\label{subsection-rpq-PRI}
%%%%%%%%%%%%%%%
In this subsection, for a commutative Noetherian ring $R$, we show various properties of the map
${\rm r_{\p\q}}$ from $\tors (\kappa(\p)Q)$ to $\tors(\kappa(\q)Q)$.
The following result was shown in \cite[Corollary 4.8(a), Proposition 4.17(b)]{Iyama-Kimura} under the assumption that $R$ contains a field, and now we are able to drop it.

%%%%%%%%%%%%%%%
%%%%%%%%%%%%%%%
\begin{prop}\label{prop-r-ftors-iso}
Let $R$ be a commutative Noetherian ring, and $Q$ a finite acyclic quiver.
Then for prime ideals $\p \supseteq \q$ of $R$, the following statements hold.
\begin{enumerate}[{\rm (a)}]
\item The map ${\rm r}_{\p\q}$ gives an isomorphism of posets:
\[{\rm r_{\p\q}}:\ftors(\kappa(\p)Q) \xto{\simeq} \ftors(\kappa(\q)Q).\]
\item Let $\alpha, \beta\in\Phi^{\rm rS}$. %then ${\rm r_{\p\q}}(\gen M_{\alpha}^{\kappa(\p)})=\gen M_{\alpha}^{\kappa(\q)}$ holds. In particular, 
Then $M_{\beta}^{\kappa(\p)}\in\gen M_{\alpha}^{\kappa(\p)}$ implies $M_{\beta}^{\kappa(\q)}\in\gen M_{\alpha}^{\kappa(\q)}$.
\end{enumerate}
\end{prop}
%%%%%%%%%%%%%%%
Our proof shows that, for the set $\Clus(Q)$ of clusters \cite[Definition 3.1]{IK2}, the isomorphism ${\rm r}_{\p\q}$ commutes with the bijections $\Clus(Q) \simeq \ftors(\kappa(\p)Q)$ and $\Clus(Q) \simeq \ftors(\kappa(\q)Q)$ due to Ingalls-Thomas \cite{Ingalls-Thomas}.

%%%%%%%%%%%%%%%
\begin{proof}
%\new{Recall that, for a Noetherian $R$-algebra $\L$, a \emph{silting $\L$-module} is a module $M\in\mod\L$ such that there exists a two-term silting complex $P$ with $\add M = \add H^0(P)$. We denote by $\siltm\L$ the set of additive equivalent classes of silting $\L$-modules.
(a) By \cite[Corollary 4.8]{Iyama-Kimura}, the squares of the following diagram is commutative and the map ${\rm r}_{\p\q}$ restricts to functorially finite torsion classes. %from $\tors\kappa(\p)Q$ to $\tors\kappa(\q)Q$ 
%Moreover the left triangle is commutative by the construction of $M^{R_{\p}}_{-}$.
\[\xymatrix@R0.5em@C3em{
 & \twosilt R_\p Q\ar[rr]^{\kappa(\p)\otimes_{R_{\p}}(-)}\ar[dd]^{(-)_{\q}} && \twosilt\kappa(\p)Q\ar[rr]^{\gen H^0(-)} &&\ftors\kappa(\p)Q\,\ar@{^{(}->}[r] \ar[dd]^{r_{\p\q}} & \tors\kappa(\p)Q \ar[dd]^{r_{\p\q}} \\ \\
%\sP_{\rm cl}(\Phi_{\geq -1}) \ar[ur]^{M^{R_{\p}}_{-}} \ar[dr]_{M^{R_{\q}}_{-}}\\
 & \twosilt R_\q Q\ar[rr]^{\kappa(\q)\otimes_{R_{\q}}(-)} && \twosilt\kappa(\q)Q\ar[rr]^{\gen H^0(-)} &&\ftors\kappa(\q)Q\,\ar@{^{(}->}[r] & \tors\kappa(\q)Q
}\]
% \ar@{^{(}->}[r] inclusion command
%Since $R_{\p}Q, R_{\q}Q$ are semi-perfect rings and by \cite[Proposition 4.3]{Iyama-Kimura}, $\kappa(\p)\otimes_{R_{\p}}(-)$ and $\kappa(\q)\otimes_{R_{\q}}(-)$, $\gen$ are isomorphisms of posets.
% by \cite[Proposition 4.3]{Iyama-Kimura}.
%So $\gen$ is also an isomorphisms of posets by \cite{Adachi-Iyama-Reiten}.
Since $R_\p Q$ and $R_\q Q$ are semi-perfect rings, \cite[Proposition 4.3]{Iyama-Kimura} implies that the left and the middle horizontal maps are isomorphisms of posets.
By \cite[Lemma 3.8(c)]{IK2}, $(-)_{\q}$ is an isomorphism of posets.
Thus the assertion holds.

(b)
By \cite[Theorem 3.4(a)]{IK2}, $M_{\alpha}^{R_\p}$ is a presilting $R_\p Q$-module. Thus $M_\beta^{R_\p}\in\gen M_\alpha^{R_\p}$ holds by \cite[Theorem 4.6(1)]{Iyama-Kimura}. Applying $\kappa(\q)\otimes_{R_\p}-$, we have $M_\beta^{\kappa(\q)}\in\gen M_\alpha^{\kappa(\q)}$.
%Since $(-)_{\q}$ is an embedding of posets, it is enough to show that $(-)_{\q}$ is surjective.
%$\siltm kQ =  \stilt kQ$ holds for a field $k$, and $\kappa(\p)\otimes_{R_{\p}}(-) \circ M^{R_{\p}}_{-}$ and $\kappa(\q)\otimes_{R_{\q}}(-) \circ M^{R_{\q}}_{-}$ are bijections by Proposition \ref{prop-cluster-stilt}(b).
%}
%bijections $\stilt\kappa(\p)Q\simeq\ftors\kappa(\p)Q$ and $\stilt\kappa(\p)Q\simeq\ftors\kappa(\p)Q$ given by $M\mapsto\gen M$.}
%Let $k\subset R$ be a field. Since $RQ \simeq R\otimes_k kQ$ holds and any simple $kQ$-module is $k$-simple, the assertion directly follows from \cite[Proposition 4.17]{Iyama-Kimura}.
\end{proof}
%%%%%%%%%%%%%%%

We use the following proposition to prove Proposition \ref{lem-rT-lift}.
%%%%%%%%%%%%%%%
\begin{prop}\label{prop-r-not-ftors}
Let $R$ be a commutative Noetherian ring, and $Q$ an extended Dynkin quiver. Then for prime ideals $\p\supseteq \q$ of $R$ and a torsion class $\cT\in\tors \kappa(\p)Q$, $\cT$ is functorially finite if and only if ${\rm r}_{\p\q}(\cT)$ is functorially finite. 
Thus ${\rm r}_{\p\q}$ gives a map
\[{\rm r_{\p\q}}:\ntors(\kappa(\p)Q) \to \ntors(\kappa(\q)Q).\]
\end{prop}
%%%%%
\begin{proof}
The ``only if'' part follows from Proposition \ref{prop-r-ftors-iso}(a).
To prove the ``if'' part, fix $\cT\in\tors \kappa(\p)Q$ which is not functorially finite.
By \cite[Theorem 3.1]{DIJ}, there is a strictly decreasing chain of functorially finite torsion classes $\mod kQ = \cT^0 \supset \cT^1 \supset \cT^2 \supset \dots \supset \cT$, and there is a strictly ascending chain of functorially finite torsion classes $0 = \cT'^0 \subset \cT'^1 \subset  \dots \subset \cT$.
By Proposition \ref{prop-r-ftors-iso}(a), each inclusion of $\cT^i$ (respectively, $\cT'^i$) is strict after applying ${\rm r}_{\p\q}$.
In particular, ${\rm r}_{\p\q}(\cT)$ is upper infinite and lower infinite.
By Theorem \ref{thm-ch-ftors-ext-Dynkin}(i) and (ii), ${\rm r}_{\p\q}(\cT)$ is not functorially finite, as desired.
\end{proof}
%%%%%%%%%%%%%%%
%\old{For $M\in\mod kQ$ and an $k$-algebra $S$, we write
%\[M_S:=S\otimes_kM\in\mod SQ.\]}
Recall that exceptional $kQ$-modules over field $k$ are parametrized by real Schur roots \eqref{prop-cluster-stilt}.
%\old{Recall that exceptional modules are parametrized by real Schur roots \eqref{prop-cluster-stilt}.} We use the following lemma in the proof of Theorem \ref{thm-compatible-extDynkin}.
%%%%%%%%%%%%%%%
\begin{prop}\label{lem-rT-lift}
Let $(R,\m,k)$ be a discrete valuation ring with quotient field $K$, $Q$ an extended Dynkin quiver, $\cT\in\tors kQ$ and $\alpha\in\Phi^{\rm rS}$. Then $M_\alpha^{k}\in\cT$ holds if and only if $M_\alpha^{K}\in{\rm r}_{\m0}(\cT)$ holds.
%Let $R$ be a commutative Noetherian ring, and $Q$ an extended Dynkin quiver.
%Let $\p\supsetneq\q$ be prime ideals of $R$ such that $\p$ has height one and $R_\p$ is regular.
%For each $\alpha\in\Phi^{\rm rS}$, $M_\alpha^{\kappa(\p)}\in\cT$ holds if and only if $M_\alpha^{\kappa(\q)}\in{\rm r}_{\p\q}(\cT)$ holds.
%\old{Let $\cT\in\tors\kappa(\p)Q$ and $N\in\cP^k\vee \cR_{\geq 2}^k \vee \cI^k$. Then $N_{\kappa(\p)}\in\cT$ holds if and only if $N_{\kappa(\q)}\in{\rm r}_{\p\q}(\cT)$ holds.}
\end{prop}
%%%%%%%%%%%%%%%
\begin{proof}
Let $\psi=\psi_{\m} : \tors kQ \to \tors RQ$.
The ``only if'' part is clear since $M_\alpha^{k}\in\cT$ implies $M_\alpha^{R}\in\psi(\cT)$ and hence $M_\alpha^{K}=K\otimes_{R}M_\alpha^{R}\in K\otimes_{R}\psi(\cT)={\rm r}_{\m 0}(\cT)$.

In the rest, we prove the ``if'' part.

(Case 1) Assume $\alpha\in\Phi_{\rm pp}^{\rm rS}$. Let $\cU:=\gen M_\alpha^{k}\in\tors kQ$. 
Since $M_{\alpha}^{R}$ is a partial tilting (hence presilting) $R Q$-module, ${\rm r}_{\m 0}(\cU)=\gen M_\alpha^{K}\subset {\rm r}_{\m 0}(\cT)$ holds by \cite[Corollary 4.8]{Iyama-Kimura}.
Since ${\rm r}_{\m 0}(\cU)$ and ${\rm r}_{\m 0}(\cT)$ contain $M_{\alpha}^{K}\in\cP^{K}$, they are functorially finite by Theorem \ref{thm-ch-ftors-ext-Dynkin}(iii)$\Rightarrow$(i). Thus $\cU$ and $\cT$ are functorially finite by Proposition \ref{prop-r-not-ftors}. By Proposition \ref{prop-r-ftors-iso}(a), we have $\cU\subset\cT$. Thus $M_{\alpha}^{k}\in\cT$.

(Case 2) Assume $\alpha\in\Phi_{\rm rg}^{\rm rS}\sqcup\Phi_{\rm pi}^{\rm rS}$.
By using induction on $\alpha$, we prove $M_{\alpha}^{k}\in\cT$.
Namely, for $\beta\in\Phi_{\rm rg}^{\rm rS}\sqcup\Phi_{\rm pi}^{\rm rS}$ such that $M_{\beta}^K\in{\rm r}_{\m 0}(\cT)$ and $\alpha > \beta$, we assume that $M_{\beta}^k\in\cT$.
Take $X\in\psi(\cT)$ such that $K \otimes_{R}X \simeq M_{\alpha}^{K}=K\otimes_{R}M_{\alpha}^{R}$.
By Proposition \ref{prop-filt-eq}, $N:=k \otimes_{R}X\in\cT$ and $M_{\alpha}^{k}=k\otimes_{R}M_{\alpha}^{R}$ admit reverse filtration.
Thus there exists non-zero morphism $f:N\to M_{\alpha}^{k}$.
We have a short exact sequence
\begin{align*}%\label{exact-N-C}
0 \to \Im f \to M_{\alpha}^{k} \to \Cok f \to 0
\end{align*}
with $\Im f\in\cT$.
By Lemma \ref{factor of M}, there exists a decomposition $\Cok f=\bigoplus_{i=1}^\ell M_{\beta_i}^{k}$ with $\beta_i\in\Phi^{\rm rS}$.
Since $M_{\beta_i}^{k}\in\gen M_{\alpha}^{k}$, we have $M_{\beta_i}^{K}\in\gen M_{\alpha}^{K}\subset{\rm r}_{\m 0}(\cT)$ by Proposition \ref{prop-r-ftors-iso}(b).
Since $\Im f\neq 0$ implies $\alpha>\beta_i$, we have $M_{\beta_i}^{k}\in\cT$ by induction hypothesis. 
Thus $\Cok f\in\cT$ and hence $M_{\alpha}^{k}\in\cT$.
\end{proof}

%%%%%%%%%%%%%%%
%\subsection{Properties of ${\rm r_{\p\q}}$ with respect to $\cH$}\label{subsection-rpq-H}
The following proposition gives a key step in the proof of Theorem \ref{thm-compatible-extDynkin}.
%%%%%%%%%%%%%%%
\begin{thm}\label{key lemma}
Let $(R,\m,k)$ be a discrete valuation ring with quotient field $K$, $Q$ an extended Dynkin quiver, and $\cT\in\tors kQ$.
%$(\cX^{\m},\cX^0)$ a compatible element of $RQ$.
If $X\in\mod RQ$ satisfies that $k\otimes_RX\in\cH^k$ and $0\neq K\otimes_RX\in{\rm r}_{\m0}(\cT)\cap\cH^K$, then $k\otimes_RX\in\cT$ holds.
\end{thm}
To prove this, we need the following preparation.
%%%%%%%%%%%%%%%
\begin{prop}\label{prop-Xp-HI}
Let $R$ be a commutative Noetherian ring, $\p\supseteq \q$ in $\Spec R$ and $Q$ an extended Dynkin quiver.
Let $\cT\in\tors \kappa(\p)Q$ such that $\cT\cap \cP^{\kappa(\p)} \neq 0$ and ${\rm r}_{\p\q}(\cT) \cap \cH^{\kappa(\q)} \neq 0$.
Then we have $\cH^{\kappa(\p)}\vee \cI^{\kappa(\p)} \subseteq \cT$.
\end{prop}
%%%%%%%%%%%%%%%
%To show Proposition \ref{prop-Xp-HI}, we need Lemmas \ref{lem-pp-hom-non-zero}, \ref{lem-tors-pp-inf} and Proposition \ref{prop-genP-H-I}.
%%%%%%%%%%%%%%%
\begin{proof}
Since $\cT\cap \cP^{\kappa(\p)} \neq 0$, Proposition \ref{prop-up-down-finite}(a) and Theorem \ref{thm-ch-ftors-ext-Dynkin}(i) and (ii) imply that $\cT$ is upper finite and functorially finite.
Thus by Proposition \ref{prop-r-ftors-iso}(a), ${\rm r}_{\p\q}(\cT)$ is also functorially finite.
Since ${\rm r}_{\p\q}(\cT) \cap \cH^{\kappa(\q)} \neq 0$, Lemma \ref{lem-pp-hom-non-zero}(b) implies ${\rm r}_{\p\q}(\cT)^{\perp}\cap\cI^{\kappa(\q)}= 0$.
Thus by Proposition \ref{prop-up-down-finite}(b), ${\rm r}_{\p\q}(\cT)$ is lower infinite. By Lemma \ref{lem-tors-pp-inf}(a), there is an infinite strictly descending chain ${\rm r}_{\p\q}(\cT)=\cT^0\supset\cT^1\supset\cT^2\supset\cdots$ in $\ftors\kappa(\q)Q$, and by Proposition \ref{prop-r-ftors-iso}(a), $\cT$ is also lower infinite. 
%By Lemma \ref{lem-tors-pp-inf}\new{(b)}, $\ind(\cT \cap \cP^{\kappa(\p)})$ is an infinite set.
%So $\cT$ contains a $\tau^-$-sincere indecomposable preprojective module, since all but finitely many indecomposable preprojective $\kappa(\p)Q$-modules are sincere by \cite[Proposition IX.5.6]{ASS}.
By Theorem \ref{upper not lower}(a)(i)$\Rightarrow$(iii), $\cT$ contains $\cH^{\kappa(\p)} \vee \cI^{\kappa(\p)}$.
\end{proof}

%\new{
%As an application, we obtain the following key step in the proof of Theorem \ref{thm-compatible-extDynkin}.
Now we are ready to prove Theorem \ref{key lemma}.

\begin{proof}[Proof of Theorem \ref{key lemma}]
Since $K\otimes_R X\in{\rm r}_{\m 0}(\cT)=K\otimes_R \psi(\cT)$, there exists $Y\in\psi(\cT)$ such that $K\otimes_R Y \simeq K\otimes_R X$.
Since $R$ is a discrete valuation ring, we can assume $Y\in\proj R$ by replacing $Y$ by $Y/\{y\in Y\mid (Ry)_0=0\}$.

Applying Proposition \ref{prop-filt-eq} to $X,Y\in\mod RQ$, it follows that $M:=k\otimes_RX$ and $N:=k\otimes_RY$ admit reverse filtrations. 
Since $M\in\cH^k$ holds by our assumption, Lemma \ref{reverse lemma}(c) implies that either $\sT(M)=\sT(N)$ or $N^{\rm pp}\neq0$ holds.
In the former case, $M\in\sT(N)\subset\cT$ holds, as desired.
In the latter case, since $0\neq N^{\rm pp}\in\cT\cap \cP^k$ and $0\neq K\otimes_RX\in{\rm r}_{\m 0}(\cT) \cap \cH^K$ hold, we have $\cH^k\subset\cT$ by  Proposition \ref{prop-Xp-HI}. Thus $M\in\cH^k\subset\cT$ holds, as desired.
\end{proof}
%}

%%%%%%%%%%%%%%%
%%%%%%%%%%%%%%%
\subsection{Proof of Theorem \ref{thm-compatible-extDynkin}}\label{subsec-proof}
To prove Theorem \ref{thm-compatible-extDynkin}, we need some preparations.
%%%%%%%%%%%%%%%
\begin{lem}
Let $R$ be a commutative Noetherian ring and $\L$ a Noetherian $R$-algebra. 
Then a compatible element $\cX=(\cX^\p)_\p\in\bT_R(\L)$ belongs to $\Im\Phi_{\rm t}$ if and only if the following condition {\rm(C$_\q$)} is satisfied for each $\q\in\Spec R$.
\begin{enumerate}
\item[{\rm(C$_\q$)}] For each $X'\in\cX^\q\cap\brick (\kappa(\q)\otimes_R \L)$, there exists $X\in\mod\L$ such that $X_\q\simeq X'$ as $\L_\q$-modules and $\kappa(\p)\otimes_R X\in\cX^\p$ for each $\p\in\Spec R$.
\end{enumerate}
\end{lem}

\begin{proof}
Recall that we have a map
\[\Psi_{\rm t}:\bT_R(\L)\to\tors\L,\ \cX=(\cX^\p)_\p\mapsto\Psi(\cX):=\{X\in\mod\L\mid \forall\p\in\Spec R,\ \kappa(\p)\otimes_RX\in\cX^\p\}\]
satisfying $\Psi_{\rm t}\circ\Phi_{\rm t}=1_{\tors\L}$ \cite[Theorem 3.16]{Iyama-Kimura}. For $\cX\in\bT_R(\L)$, let $\cY=(\cY^\p)_\p:=\Phi_{\rm t}\circ\Psi_{\rm t}(\cX)\in\bT_R(\L)$.
Then $\cX\in\Im\Phi_{\rm t}$ if and only if $\cX=\cY$ if and only if $\cX^\p\subset\cY^\p$ for each $\p\in\Spec R$ if and only if $\cX^\q\cap\brick(\kappa(\q)\otimes_R \L)\subset\cY^\q$ for each $\q\in\Spec R$, where the last equivalence follows from \eqref{brick generate tors}.
The last condition means that (C$_\q$) holds for each $\q\in\Spec R$.
\end{proof}

To prove the condition (C$_\q$), the following observation is useful.

\begin{prop}{\cite[Proposition 3.34]{Iyama-Kimura}}\label{correction}
Let $R$ be a commutative Noetherian ring, $\L$ a Noetherian $R$-algebra, $\cX=(\cX^\p)_\p\in\bT_R(\L)$ a compatible element and $V$ a finite Zariski closed subset of $\Spec R$.
If $X\in\mod\L$ satisfies $\kappa(\p)\otimes_RX\in\cX^\p$ for each $\p\in V^{\rm c}$, then there exists a submodule $Y$ of $X$ such that $\kappa(\p)\otimes_RY\in\cX^\p$ for each $\p\in\Spec R$ and $X_\p\simeq Y_\p$ for all $\p\in V^{\rm c}$.
\end{prop}

For example, (C$_\q$) is automatic if $V(\q)$ is a finite set. In fact, for each $X'\in\cX^\q$, there exists $X\in\mod\L$ such that $\supp X\subseteq V(\q)$ and $X_\q\simeq X'$ \cite[Lemma 2.3(b)]{Iyama-Kimura}. Applying Proposition \ref{correction} to $V:=V(\q)\setminus\{\q\}$, we obtain $Y\in\mod\L$ satisfying the desired condition in (C$_\q$).

\begin{proof}[Proof of Theorem \ref{thm-compatible-extDynkin}.]
Let $\cX=(\cX^\p)_\p\in\bT_R(RQ)$ be a compatible element. We will show that (C$_\q$) holds for each $\q\in\Spec R$. As remarked above, this is automatic if $V(\q)$ is finite.
By our assumption (II), we have $\dim R\le 2$ and $V(\q)$ is finite for all non-minimal prime ideals $\q$. Thus it suffices to prove (C$_\q$) for each minimal prime ideal $\q$ of $R$.
Fix $X'\in\cX^{\q}\cap\brick\kappa(\q)Q$. It suffices to show that there exists $X\in\mod RQ$ satisfying the following conditions.
\begin{enumerate}[\rm(i)]
\item $X_\q\simeq X'$ and $\supp X\subset V(\q)$.
\item[(ii)] There exists a Zariski open neighborhood $\q\in U\subset\Spec R$ satisfying $\kappa(\p)\otimes_R X\in\cX^\p$ for all $\p\in U$ of height one.
\end{enumerate}
In fact, $V:=U^{\rm c}\cap V(\q)$ is a finite Zariski closed subset of $\Spec R$ by our assumption (II), and applying Proposition \ref{correction} to $X$, we obtain $Y\in\mod\L$ satisfying the desired condition (C$_\q$).

We divide into two cases. First, we assume $X'\in\excep\kappa(\q)Q$. By \eqref{prop-cluster-stilt}, there exists $\alpha\in\Phi^{\rm rS}$ such that $X'=M_\alpha^{\kappa(\q)}$. Then $X:=M_\alpha^{R/\q}\in\mod RQ$ satisfies the condition (i). It also satisfies the condition (ii) for $U:=\Spec R$. In fact, take any $\p\in\Spec R$ with height one. If $\p\notin V(\q)$, then $\kappa(\p)\otimes_RX=0\in\cX^\p$ holds. If $\p\in V(\q)$, then $M_\alpha^{\kappa(\q)}=X'\in\cX^{\q}\subset{\rm r}_{\p\q}(\cX^\p)$ holds. Since $R$ satisfies ($R_1$), Proposition \ref{lem-rT-lift} implies $\kappa(\p)\otimes_RX=M_\alpha^{\kappa(\p)}\in\cX^\p$, as desired.

%\old{$X''\in\mod kQ$ such that $R_\q\otimes_k X'' \simeq X'\in\cX^{\q}\subset{\rm r}_{\p\q}(\cX^\p)$. By our assumption (I) and Proposition \ref{lem-rT-lift}, $\kappa(\p)\otimes_k X''\in\cX^\p$ holds for any $\p\in\Spec R$ with height one. Thus $X:=(R/\q)\otimes_k X''\in\mod RQ$ and $U:=\Spec R$ satisfies the conditions (i) and (ii). In fact, $\kappa(\p)\otimes_RX$ is $\kappa(\p)\otimes_kX''\in\cX^\p$ if $\p\in V(\q)$ and $0\in\cX^\p$ otherwise.}

In the rest, we assume $X'\notin\excep\kappa(\q)Q$. Then $X'\in\cH^{\kappa(\q)}$ holds. By Lemma \ref{lem-local-dense}, there exists $X\in\mod RQ$ such that $X_\q\simeq X'$. We can assume $\supp X\subset V(\q)$ by replacing $X$ by $X/\q X$. Thus (i) is satisfied.
By our assumption (I) and Lemma \ref{lem-M-0}(b), there is an open neighborhood $\q\in U$ such that 
\begin{align}\label{kpX-Hp}
\mbox{$\kappa(\p)\otimes_R X\in\cH^{\kappa(\p)}$ for any $\p\in U$ of height one.}
\end{align}
Moreover, since $R_\q$ is a field, we can assume that $X_\p\in\proj R_\p$ holds for any $\p\in U$ by applying Lemma \ref{lem-ractional-NZD} for $\Lambda:=R$ and making $U$ smaller.

To prove that $U$ satisfies the condition (ii), we fix $\p\in U$ of height one.
We can apply Theorem \ref{key lemma} to $(R,X):=(R_\p,X_\p)$ and the compatible element $(\cX^{\p},\cX^0)$ for $R_\p Q$ since $\kappa(\p)\otimes_{R_\p}X_\p\in\cH^{\kappa(\p)}$ and $X_\q\simeq X'\in\cH^{\kappa(\q)}$ hold by our assumptions. Consequently, we obtain $\kappa(\p)\otimes_{R_\p}X_\p\in\cX^\p$, as desired.
%Since $(\cX^\p)_\p$ is compatible, there exists $Y\in\psi_\p(\cX^\p)$ such that $Y_\q \simeq X'\simeq X_\q$. Since $R_\p$ is a discrete valuation ring, we can assume $Y\in\proj R_\p$ by replacing $Y$ by $Y/\{y\in Y\mid (Ry)_\q=0\}$.
%\new{Applying Proposition \ref{prop-filt-eq} to $X_\p,Y\in\mod R_\p Q$, it follows that $M:=\kappa(\p)\otimes_{R_\p}X_\p$ and $N:=\kappa(\p)\otimes_{R_\p}Y$ admit reverse filtrations.  Since $M\in\cH^{\kappa(\p)}$ holds by \eqref{kpX-Hp}, Lemma \ref{reverse lemma}(c) implies that either $\sT(M)=\sT(N)$ or $N^{\rm pp}\neq0$ holds. In the former case, $M\in\sT(N)\subset\cX^\p$ holds, and hence (ii) is satisfied. In the latter case, since $0\neq N^{\rm pp}\in\cX^\p\cap \cP^{\kappa(\p)}$ and $0\neq X'\in\cX^\q \cap \cH^{R_\q}\subset{\rm r}_{\p\q}(\cX^\p)\cap \cH^{R_\q}$ hold, we have $\cH^{\kappa(\p)}\stackrel{\rm}{\subset}\cX^\p$ by  Proposition \ref{prop-Xp-HI}. Thus $M\in\cH^{\kappa(\p)}\subset\cX^\p$ holds, and hence (ii) is satisfied.}
\end{proof}

%%%%%%%%%%%%%%%
%%%%%%%%%%%%%%%
\section{Torsion classes for Kronecker quiver over Dedekind domains}\label{section-tors-Kro}
In this section, let $\begin{tikzpicture}[baseline=-3]
    % Define the two vertices
    \node(A) at (0, 0) {$Q=\ 1$};
    \node(B) at (2, 0) {$2$};
    % Draw the arrows (edges)
    %\draw[->] ([yshift=10pt] 1) -- ([yshift=10pt] 2) node[midway, above] {$f_1$};
    %\draw[->] ([yshift=-10pt] 1) -- ([yshift=-10pt] 2) node[midway, below] {$f_2$};
    \draw[->] ([yshift=2pt ]A.east) --  node[midway, above] {\scriptsize$a$} ([yshift=2pt ]B.west);
    \draw[->] ([yshift=-2pt ]A.east) -- node[midway, below] {\scriptsize$b$} ([yshift=-2pt ]B.west) ;
\end{tikzpicture}$
be the Kronecker quiver. For each field $k$, $\ftors kQ$ has the following Hasse quiver, where $M_{i,j}^k$ shows the torsion class $\gen M^k_{i\alpha_1+j\alpha_2}$ generated by the exceptional $kQ$-module $M^k_{i\alpha_1+j\alpha_2}$ corresponding to $i\alpha_1+j\alpha_2\in\Phi^{\rm rS}$ given in \eqref{prop-cluster-stilt}:
\begin{align}\label{hasse}
\xymatrix@R0em@C1.5em{
&&&&M_{0,1}^k\ar[drrrr]\\
\mod kQ\ar[dr]\ar[urrrr]&&&&&&&&0.\\
&M_{1,2}^k\ar[r]&M_{2,3}^k\ar[r]&M_{3,4}^k\ar[r]&\cdots\ar[r]& M_{3,2}^k\ar[r]&M_{2,1}^k\ar[r]&M_{1,0}^k\ar[ur]
}\end{align}
Now let $R$ be a Dedekind domain. By Theorem \ref{thm-compatible-extDynkin}, we have an isomorphism of posets 
\[\tors RQ \simeq \Big\{(\cX^\p)_\p\in \prod_{\p\in\Spec R}\tors\kappa(\p)Q \mid {}^{\forall}\m\in \Spec R\setminus\{0\},\,{\rm r}_{\m0}(\cX^\m) \supseteq \cX^0\Big\}, \cT\mapsto(\kappa(\p)\otimes_R\cT)_{\p\in\Spec R}.\]
The aim of this section is to give an explicit description of the map ${\rm r}_{\m0}:\tors\kappa(\m)Q\to\tors\kappa(0)Q$. 
By Proposition \ref{prop-r-ftors-iso}(a), it restricts to an isomorphism ${\rm r}_{\m0}:\ftors\kappa(\m)Q\simeq\ftors\kappa(0)Q$ of posets with the Hasse quiver \eqref{hasse}. 
Thus it remains to describe the map ${\rm r}_{\m0}:\ntors\kappa(\m)Q\to\ntors\kappa(0)Q$ given in Proposition \ref{prop-r-not-ftors}.

\subsection{The ${\rm r}_{\m0}$ map for Kronecker quiver }
%Let $R$ be a commutative ring.
For a field $k$, let
\[\irr k[x]:=\{\mbox{irreducible polynomials in $k[x]$}\}/\sim,\]
where we write $f\sim g$ if $f=\lambda g$ holds for some $\lambda \in k^{\times}$. Let
\[|\bP_k^1|:=\irr k[x] \sqcup \{x^{-1}\}.\]
For $f \in k[x]$, we define a $kQ$-module $M_f^k$ by
\begin{align*}
\begin{tikzpicture}[baseline=0]
    % Define the two vertices
    \node(A) at (0, 0) {$M_f^k:=\ \displaystyle k[x]/(f)$};
    \node(B) at (3, 0) {$\displaystyle k[x]/(f)$};
    % Draw the arrows (edges)
    %\draw[->] ([yshift=10pt] 1) -- ([yshift=10pt] 2) node[midway, above] {$f_1$};
    %\draw[->] ([yshift=-10pt] 1) -- ([yshift=-10pt] 2) node[midway, below] {$f_2$};
    \draw[->] ([yshift=2pt ]A.east) --  node[midway, above] {\scriptsize$1$} ([yshift=2pt ]B.west);
    \draw[->] ([yshift=-2pt ]A.east) -- node[midway, below] {\scriptsize$x$} ([yshift=-2pt ]B.west) ;
\end{tikzpicture}.
\end{align*}
Then, for the power set $\sP(|\bP_k^1|)$ of $|\bP_k^1|$, we have an isomorphism of posets
\[
\sT_k : \sP(|\bP_k^1|) \stackrel{\sim}{\longrightarrow} \ntors kQ, \quad \cS \mapsto \sT_k(\cS):=\sT(M_f^k \mid f \in \cS) \vee \cI^k.
\]
%$V(g)$ is the set of all prime ideals of $k[x]$ containing $g$. We identify $V(g)$ as a subset of $\irr k[x]$ consisting of $f$ which divides $g$.

From now on, let $(R, \m, k)$ be a discrete valuation ring with its quotient field $K$.
For each $f \in \irr K[x]$, there exists an element $r\in R$ such that $rf \in R[x]$ and $rf$ is primitive in $R[x]$ (that is, the greatest common divisor of all coefficients of $rf$ is a unit element of $R$).
Hence we denote by $\irr_R K[x]$ a complete set representative of $\irr K[x]$ consisting of primitive polynomials of $R[x]$.
The canonical surjection $\overline{(-)} : R \to k$ gives a surjection $\overline{(-)} : R[x] \to k[x]$.
%This morphism is restricted to $\overline{(-)} : \irr_R K[x] \to k[x]$. 
Let
\[|\bP_K^1| := \irr_R K[x] \sqcup\{x^{-1}\}.\]
For $g\in k[x]$, let $V(g):=\{f\in\irr k[x]\mid g\in(f)\}$.
For $f(x)=r_0 + r_1x + \dots + r_dx^d \in R[x]$ with $r_d\neq 0$, let
\[W(f):=\left\{\begin{array}{ll}V(\ov{f}) \sqcup \{x^{-1}\}&r_d \in \m\\
V(\ov{f})&\mbox{else.}\end{array}\right.\]
Moreover $W(x^{-1}) := \{x^{-1}\}$.
We define a map
\[
{\rm r} : \sP(|\bP_k^1|) \longrightarrow \sP(|\bP_K^1|), \quad \cS \mapsto {\rm r}(\cS):=\{ f \in |\bP_K^1| \mid W(f) \subseteq \cS\}.
\]
Clearly, ${\rm r}$ is an order preserving map satisfying ${\rm r}(\emptyset)=\emptyset$ and ${\rm r}(|\bP_k^1|)=|\bP_K^1|$. One can also check that ${\rm r}$ is an embedding of posets.
%\begin{enumerate}[\rm(a)]
%\end{enumerate}
%which is clearly order preserving.
The following is the main theorem of this section.
%%%%%%%%%%%%%%%
\begin{thm}\label{thm-rT=Tr}
For each $\cS\in\sP(|\bP_k^1|)$, we have ${\rm r}_{\m 0}(\sT_k(\cS)) = \sT_K({\rm r}(\cS))$. Thus we have a commutative diagram of posets:
\[
\xymatrix@R0.5em@C3em{
\sP(|\bP_k^1|) \ar[rr]^{\rm r} \ar[dd]^{\sT_k} && \sP(|\bP_K^1|) \ar[dd]^{\sT_K} \\ \\
\ntors kQ \ar[rr]^{{\rm r}_{\m 0}} && \ntors KQ.
}
\]
\end{thm}

%\begin{rem}
%The map ${\rm r}:\sP(|\bP_k^1|) \longrightarrow \sP(|\bP_K^1|)$ can be defined
%\end{rem}
%\begin{cor}
%Let $R$ be a Dedekind domain. Then we have an isomorphism of posets
%\[\ntors RQ\simeq\Big\{(\cS^{\p})_{\p}\in\prod_{\p\in\Spec R}\sP(|\bP_{\kappa(\p)}^1|)\mid\forall\m\in\Spec R\setminus\{0\},\ {\rm r}(\cS^{\m})\supseteq\cS^{0}\Big\}.\]
%\end{cor}

\subsection{Proof of Theorem \ref{thm-rT=Tr}}
%\subsection{Regular modules over Kronecker quiver}
Let $R$ be a commutative ring.
Any $RQ$-module is presented by a 4-tuple $(M_1, M_2, f_a, f_b)$, where $M_1$ and $M_2$ are $R$-modules, and $f_a$ and $f_b$ are morphisms of $R$-modules from $M_1$ to $M_2$.

Let $R[x]$ be the ring of polynomials in one variable over $R$. For $f(x)\in R[x]$, define $RQ$-modules by
%$M_f^R$ and $N_f^R$ as follows:
\begin{align*}
M_f^R:=(R[x]/(f), R[x]/(f), 1, x)\ \mbox{ and }\ 
%= \begin{tikzpicture}[baseline=0]
%    % Define the two vertices
%    \node(A) at (0, 0) {$\displaystyle R[x]/(f)$};
%    \node(B) at (3, 0) {$\displaystyle R[x]/(f)$};
%    % Draw the arrows (edges)
%    %\draw[->] ([yshift=10pt] 1) -- ([yshift=10pt] 2) node[midway, above] {$f_1$};
%    %\draw[->] ([yshift=-10pt] 1) -- ([yshift=-10pt] 2) node[midway, below] {$f_2$};
%    \draw[->] ([yshift=2pt ]A.east) --  node[midway, above] {\small$1$} ([yshift=2pt ]B.west);
%    \draw[->] ([yshift=-2pt ]A.east) -- node[midway, below] {\small$x$} ([yshift=-2pt ]B.west) ;
%\end{tikzpicture}, \\
N_f^R:=(R[x]/(f), R[x]/(f), x, 1).
%= \begin{tikzpicture}[baseline=0]
%    % Define the two vertices
%    \node(A) at (0, 0) {$\displaystyle \frac{R[x]}{(f)}$};
%    \node(B) at (2, 0) {$\displaystyle \frac{R[x]}{(f)}$};
%    % Draw the arrows (edges)
%    %\draw[->] ([yshift=10pt] 1) -- ([yshift=10pt] 2) node[midway, above] {$f_1$};
%    %\draw[->] ([yshift=-10pt] 1) -- ([yshift=-10pt] 2) node[midway, below] {$f_2$};
%    \draw[->] ([yshift=5pt ]A.east) --  node[midway, above] {$x$} ([yshift=5pt ]B.west);
%    \draw[->] ([yshift=-5pt ]A.east) -- node[midway, below] {$1$} ([yshift=-5pt ]B.west) ;
%\end{tikzpicture}
\end{align*}
Moreover, let $M_{x^{-i}}^R:=N_{x^i}^R$ for each $i\ge0$.
We give some preliminary observations for $M_f^R$ and $N_f^R$.
%%%%%%%%%%%%%%%
\begin{lem}\label{lem-Mf=sum}
Let $k$ be a field, and $f(x)\in k[x]$ be non-zero. Take a decomposition $f\sim f_1^{\ell_1} \cdots f_m^{\ell_m}$
%be an irreducible decomposition of $f(x)\in k[x]$, that is, 
such that each $f_i$ is irreducible, $\ell_i\geq1$ and $f_i\not\sim f_j$ for each $i\neq j$.
\begin{enumerate}[\rm(a)]
\item We have $M_f^k \simeq \bigoplus_{i=1}^m M_{f_i^{\ell_i}}^k$ and $N_f^k \simeq \bigoplus_{i=1}^m N_{f_i^{\ell_i}}^k$ as $kQ$-modules.
\item $\sT(M_f^k)\vee\cI^k = \sT_k(V(f))$ holds.
\end{enumerate}
\end{lem}

%%%%%%%%%%%%%%%
%\begin{lem}\label{lem-TM=TV}
%For $g\in k[x]\setminus\{0\}$, 
%\end{lem}
%%%%%
\begin{proof}
(a) Clear from Chinese Remainder Theorem.
%Let $g(x) = g_1(x)^{\ell_1} \cdots g_m(x)^{\ell_m}$ be an irreducible decomposition of $g(x)$, that is, each $g_i(x)$ is irreducible, $g_i(x)\neq g_j(x)$ for $i\neq j$ and $\ell_i \geq 1$.

(b) We have $V(f)=\{f_1,\ldots,f_m\}$ and $M_f^k \simeq \bigoplus_{i=1}^m M_{f_i^{\ell_i}}^k$ by (a).
Thus the assertion holds.
\end{proof}
%%%%%%%%%%%%%%%

%%%%%
%\begin{proof}
%The assertion follows from an isomorphism $R[x]/(f) \simeq \bigoplus_{i=1}^m R[x]/(g_i^{\ell_i})$.
%\end{proof}
%%%%%%%%%%%%%%%
We denote by $E_d$ the identity matrix of size $d$.
\begin{lem}\label{lem-Mf=Nxf}
Let $f(x)=r_0 + r_1x + \dots + r_dx^d \in R[x]$ and $g(x):=x^df(x^{-1})=r_0x^d + r_1x^{d-1} + \dots + r_d$.
\begin{enumerate}[\rm(a)]
\item If $r_d\in R^{\times}$, then $M_f^R\simeq(R^d, R^d, E_d, A_f)$ and $N_f^R\simeq(R^d, R^d, A_f,E_d)$ as $RQ$-modules for $A_f:=\scalebox{0.7}{$\begin{bmatrix} 0 & 0 & \cdots &0& -r_0r_d^{-1}  \\ 1 & 0 & \cdots &0& -r_1r_d^{-1} \\ \vdots & \ddots & \ddots & \vdots & \vdots\\ 0 & \cdots & 1 &0& -r_{d-2}r_d^{-1} \\ 0 & \cdots & 0 & 1 & -r_{d-1}r_d^{-1} \end{bmatrix}$}$.
%\item If $r_0\in R^{\times}$, then $N_f^R\simeq(R^d, R^d, C, E_d)$ as $RQ$-modules, where $C:=\scalebox{0.7}{$\begin{bmatrix} 0 & 0 & \cdots &0& -r_0^{-1}r_d  \\ 1 & 0 & \cdots &0& -r_0^{-1}r_{d-1} \\ \vdots & \ddots & \ddots & \vdots & \vdots\\ 0 & \cdots & 1 &0& -r_0^{-1}r_2 \\ 0 & \cdots & 0 & 1 & -r_0^{-1}r_1 \end{bmatrix}$}$.
\item If $r_0, r_d\in R^{\times}$, then $M_f^R \simeq N_g^R$ as $RQ$-modules.
\item If $r_0\in R^\times$ and $r_d=0$, then $\left(R^d, R^d, \scalebox{0.7}{$\begin{bmatrix} 1 & 0 & \cdots &0 &0 \\ 0 & 1 & \cdots & 0 &0\\ \vdots & \vdots& \ddots & \vdots&\vdots \\ 0 &0& \cdots & 1 & 0 \\ 0 &0& \cdots & 0 & 0 \end{bmatrix}$},\scalebox{0.7}{$\begin{bmatrix} 0 & 0 & \cdots &0& -r_0  \\ 1 & 0 & \cdots &0& -r_1 \\ \vdots & \ddots &\ddots& \vdots &\vdots \\ 0 & \cdots & 1 &0& -r_{d-2}\\ 0 & \cdots &0& 1 & -r_{d-1}  \end{bmatrix}$}\right)\simeq N_g^R$ as $RQ$-modules.
\end{enumerate}
\end{lem}
%%%%%
\begin{proof}
(a) Since $r_d\in R^{\times}$, $R[x]/(f)$ is a free $R$-module with basis $1,x,x^2,\ldots,x^{d-1}$. Since $B$ is the representation matrix of the $R$-linear map $x:R[x]/(f)\to R[x]/(f)$, the assertion follows.
%, where the canonical basis $e_i\in R^d$ corresponds to $x^{i-1} \in R[x]/(f)$ for $i=1,2,\dots, d$.
%Then the matrix $B$ presents the morphism $x : R[x]/(f) \to R[x]/(f)$.

(b) Applying simultaneous elementary transformations of matrices, we have
\begin{align*}
 M_f^R & \stackrel{{\rm(a)}}{\simeq} \left(R^d, R^d, \scalebox{0.7}{$\begin{bmatrix} 1 & 0 & \cdots & 0 & r_1r_d^{-1}  \\ 0 & 1 & \cdots & 0 & r_2r_d^{-1} \\ \vdots & \vdots& \ddots & \vdots & \vdots\\ 0 & 0& \cdots & 1 & r_{d-1}r_d^{-1}  \\ 0 & 0 & \cdots & 0 & 1 \end{bmatrix}$}, \scalebox{0.7}{$\begin{bmatrix} 0 & 0 & \cdots &0& -r_0r_d^{-1}  \\ 1 & 0 & \cdots &0& 0 \\ \vdots & \ddots & \ddots & \vdots & \vdots\\ 0 & \cdots & 1 &0& 0 \\ 0 & \cdots & 0 & 1 & 0 \end{bmatrix}$}\right)
 \simeq \left(R^d, R^d, \scalebox{0.7}{$\begin{bmatrix} r_1r_d^{-1} & 1 & 0 & \cdots & 0  \\ r_2r_d^{-1} & 0 & 1 & \cdots & 0 \\ \vdots & \vdots & \ddots & \ddots & \vdots \\ r_{d-1}r_d^{-1} &  0 &  \cdots & 0 & 1  \\ 1 & 0 & \cdots & 0 & 0 \end{bmatrix}$}, \scalebox{0.7}{$\begin{bmatrix} -r_0r_d^{-1} & 0 & \cdots & 0  & 0 \\ 0 & 1 & \cdots & 0 &0  \\ \vdots & \vdots & \ddots & \vdots & \vdots \\ 0 & 0& \cdots & 1 & 0 \\ 0 & 0& \cdots & 0 & 1 \end{bmatrix}$}\right)\\
&\simeq \left(R^d, R^d, \scalebox{0.7}{$\begin{bmatrix} -r_0^{-1}r_1 & 1 & 0 & \cdots & 0  \\ -r_0^{-1}r_2 & 0 & 1 & \cdots & 0 \\ \vdots & \vdots & \ddots & \ddots & \vdots \\ -r_0^{-1}r_{d-1} &  0 &  \cdots & 0 & 1  \\ -r_0^{-1}r_d & 0 & \cdots & 0 & 0 \end{bmatrix}$}, E_d\right) 
\simeq \left(R^d, R^d, \scalebox{0.7}{$\begin{bmatrix} 0 & 0 & \cdots &0& -r_0^{-1}r_d  \\ 1 & 0 & \cdots &0& -r_0^{-1}r_{d-1} \\ \vdots & \ddots & \ddots & \vdots & \vdots\\ 0 & \cdots & 1 &0& -r_0^{-1}r_2 \\ 0 & \cdots & 0 & 1 & -r_0^{-1}r_1 \end{bmatrix}$}, E_d\right) \stackrel{{\rm(a)}}{\simeq} N^R_g.\end{align*}

(c) Applying simultaneous elementary transformations of matrices, the left-hand side is isomorphic to
\begin{align*}
 &\left(R^d, R^d, \scalebox{0.7}{$\begin{bmatrix} 1 & 0 & \cdots & 0 & r_1  \\ 0 & 1 & \cdots & 0 & r_2 \\ \vdots & \vdots& \ddots & \vdots & \vdots\\ 0 & 0& \cdots & 1 & r_{d-1}  \\ 0 & 0 & \cdots & 0 & 0 \end{bmatrix}$}, \scalebox{0.7}{$\begin{bmatrix} 0 & 0 & \cdots &0& -r_0  \\ 1 & 0 & \cdots &0& 0 \\ \vdots & \ddots & \ddots & \vdots & \vdots\\ 0 & \cdots & 1 &0& 0 \\ 0 & \cdots & 0 & 1 & 0 \end{bmatrix}$}\right)
 \simeq \left(R^d, R^d, \scalebox{0.7}{$\begin{bmatrix} r_1 & 1 & 0 & \cdots & 0  \\ r_2 & 0 & 1 & \cdots & 0 \\ \vdots & \vdots & \ddots & \ddots & \vdots \\ r_{d-1} &  0 &  \cdots & 0 & 1  \\ 0 & 0 & \cdots & 0 & 0 \end{bmatrix}$}, \scalebox{0.7}{$\begin{bmatrix} -r_0 & 0 & \cdots & 0  & 0 \\ 0 & 1 & \cdots & 0 &0  \\ \vdots & \vdots & \ddots & \vdots & \vdots \\ 0 & 0& \cdots & 1 & 0 \\ 0 & 0& \cdots & 0 & 1 \end{bmatrix}$}\right)\\
&\simeq \left(R^d, R^d, \scalebox{0.7}{$\begin{bmatrix} -r_0^{-1}r_1 & 1 & 0 & \cdots & 0  \\ -r_0^{-1}r_2 & 0 & 1 & \cdots & 0 \\ \vdots & \vdots & \ddots & \ddots & \vdots \\ -r_0^{-1}r_{d-1} &  0 &  \cdots & 0 & 1  \\ 0 & 0 & \cdots & 0 & 0 \end{bmatrix}$}, E_d\right) 
\simeq \left(R^d, R^d, \scalebox{0.7}{$\begin{bmatrix} 0 & 0 & \cdots &0& 0  \\ 1 & 0 & \cdots &0& -r_0^{-1}r_{d-1} \\ \vdots & \ddots & \ddots & \vdots & \vdots\\ 0 & \cdots & 1 &0& -r_0^{-1}r_2 \\ 0 & \cdots & 0 & 1 & -r_0^{-1}r_1 \end{bmatrix}$}, E_d\right) \stackrel{{\rm(a)}}{\simeq} N^R_g.\qedhere
\end{align*}
\end{proof}
%%%%%%%%%%%%%%%
%%%%%%%%%%%%%%%
%Now let $R$ be a commutative ring, and $\m$ a maximal ideal of $R$. Consider the canonical surjection $\overline{(-)}:R\to k:=R/\m$. 
We need one more preparation. For $f(x)=r_0 + r_1x + \dots + r_dx^d \in R[x]$ with $r_d\neq 0$, let
\[
L^{R}_f : = \left(R^d, R^d, \scalebox{0.7}{$\begin{bmatrix} 1 & 0 & \cdots &0 &0 \\ 0 & 1 & \cdots & 0 &0\\ \vdots & \vdots& \ddots & \vdots&\vdots \\ 0 &0& \cdots & 1 & 0 \\ 0 &0& \cdots & 0 & r_d \end{bmatrix}$},\scalebox{0.7}{$\begin{bmatrix} 0 & 0 & \cdots &0& -r_0  \\ 1 & 0 & \cdots &0& -r_1 \\ \vdots & \ddots &\ddots& \vdots &\vdots \\ 0 & \cdots & 1 &0& -r_{d-2}\\ 0 & \cdots &0& 1 & -r_{d-1}  \end{bmatrix}$}\right).
\]
%%%%%%%%%%%%%%%
\begin{lem}\label{lem-M'}
Let $R$ be a commutative ring, $k$ a field, and $\overline{(-)}:R\to k$ a morphism of rings.
Let $f(x)=r_0 + r_1x + \dots + r_dx^d \in R[x]$ such that $r_d\neq0$ and $\overline{f}\neq0$.
%Then the following assertions hold.
\begin{enumerate}[{\rm (a)}]
%	\item If $R$ is a domain with a quotient field $K$, then $ K \otimes_R \tilde{M}^{R}_f \simeq M_f^K$ as $KQ$-modules.
	\item $ k \otimes_R L^{R}_f \simeq M_{\ov{f}}^k \oplus M_{x^{j-d}}^k\in\cR^k$ as $kQ$-modules, where $j$ is the maximal integer satisfying $\ov{r_j}\neq 0$.
	\item $\sT(k \otimes_R L^{R}_f )\vee\cI^k = \sT_k(W(f))$.
\end{enumerate}
\end{lem}
%%%%%
\begin{proof}
(a) 
Let $i$ be the minimal integer such that $\ov{r_i} \neq 0$.
Then $0 \leq i \leq j \leq d$ and $\ov{f} = \ov{r_i}x^i + \ov{r_{i+1}}x^{i+1} + \dots + \ov{r_j}x^j$ hold. 
Let $g(x) :=x^d \ov{f}(x^{-1}) = \ov{r_i}x^{d-i} + \ov{r_{i+1}}x^{d-i-1} + \dots + \ov{r_j}x^{d-j}$ and
\[X:=\left(k^{d-i}, k^{d-i}, \scalebox{0.7}{$\begin{bmatrix} 1 & 0 & \cdots &0 &0 \\ 0 & 1 & \cdots & 0 &0\\ \vdots & \vdots& \ddots & \vdots&\vdots \\ 0 &0& \cdots & 1 & 0 \\ 0 &0& \cdots & 0 & 0 \end{bmatrix}$},\scalebox{0.7}{$\begin{bmatrix} 0 & 0 & \cdots &0& -\overline{r_i}  \\ 1 & 0 & \cdots &0& -\overline{r_{i+1}} \\ \vdots & \ddots &\ddots& \vdots &\vdots \\ 0 & \cdots & 1 &0& -\overline{r_{d-2}}\\ 0 & \cdots &0& 1 & -\overline{r_{d-1}}  \end{bmatrix}$}\right)\stackrel{{\rm\ref{lem-Mf=Nxf}(c)}}{\simeq} N_g^k.\]
%Then $X$ as $kQ$-modules  by Lemma .
The map $k^{d-i}\oplus k^{d-i}\to k^d\oplus k^d$, $(u,v)\mapsto\left(\scalebox{0.7}{$\begin{bmatrix}0\\ u\end{bmatrix}$},\scalebox{0.7}{$\begin{bmatrix}0\\ v\end{bmatrix}$}\right)$ gives an injective morphism $X\to k\otimes_RL^{R}_f$ of $kQ$-modules, and we have an exact sequence
\begin{align}\label{ses-M'}
0 \to N_g^k \to k\otimes_RL^{R}_f \to M_{x^i}^k \to 0
\end{align}
of $kQ$-modules. Since $-\overline{r_i}\neq0$, the second matrix of $N_g^k$ is invertible. Thus the $kQ$-modules $N_g^k$ and $M_{x^i}^k$ belong to different tubes in $\cR^k$, and hence $\Ext^1_{kQ}(M_{x^i}^k,N_g^k)=0$ holds. Thus the sequence \eqref{ses-M'} splits. Since $g=x^{d-j}\cdot x^{j}\ov{f}(x^{-1})$ holds and $x^{d-j}$ and $x^j\overline{f}(x^{-1})$ do not have a common factor, we have
\begin{align}\label{first}
k\otimes_RL^{R}_f \simeq N_g^k \oplus M_{x^i}^k\stackrel{{\rm\ref{lem-Mf=sum}(a)}}{\simeq} N_{x^{d-j}}^k \oplus N_{x^j\ov{f}(x^{-1})}^k \oplus M_{x^i}^k.
\end{align}
%as $kQ$-modules.
Since $h(x):=x^{-i}\ov{f}=\ov{r_i} + \ov{r_{i+1}}x+\dots+\ov{r_{j}}x^{j-i}$ satisfies $x^{j-i}h(x^{-1})=x^j\ov{f}(x^{-1})$, we have
\begin{align}\label{second}
N_{x^j\ov{f}(x^{-1})}^k=N_{x^{j-i}h(x^{-1})}^k\stackrel{{\rm\ref{lem-Mf=Nxf}(b)}}{\simeq}M_h^k=M_{x^{-i}\ov{f}}^k.
\end{align}
%as $kQ$-modules.
Thus we have 
\[k\otimes_RL^{R}_f \stackrel{\eqref{first}\eqref{second}}{\simeq} M_{x^{j-d}}^k \oplus M_{x^{-i}\ov{f}}^k \oplus M_{x^i}^k \stackrel{{\rm\ref{lem-Mf=sum}(a)}}{\simeq} M_{x^{j-d}}^k \oplus M_{\ov{f}}^k.\]
%as desired.

(b) $\sT(k \otimes_RL^{R}_f )\vee\cI^k \stackrel{{\rm (a)}}{=} \sT(M_{\ov{f}}^k \oplus M_{x^{j-d}}^k)\vee\cI^k\stackrel{{\rm\ref{lem-Mf=sum}(b)}}{=}\sT_k(V(\ov{f}))\vee\sT(M_{x^{j-d}}^k)=\sT_k(W(f))$ holds.
\end{proof}
%%%%%%%%%%%%%%%

%\subsection{Proof of Theorem \ref{thm-rT=Tr}}

Now we are ready to prove Theorem \ref{thm-rT=Tr}.
%%%%%%%%%%%%%%%

\begin{proof}[Proof of Theorem \ref{thm-rT=Tr}]
Fix $\cS \in \sP(|\bP_k^1|)$.

(I)
We prove ${\rm r}_{\m 0}(\sT_k(\cS)) \supseteq \sT_K({\rm r}(\cS))$.

Since $\sT_k(\cS)\supseteq\cI^k$, we have ${\rm r}_{\m,0}(\sT_k(\cS))\supseteq\cI^K$ by Propositions \ref{lem-rT-lift} and \ref{prop-tube-k-K}.
It suffices to show $M_f^K\in{\rm r}_{\m 0}(\sT_k(\cS))$ for each $f\in{\rm r}(\cS)$. 
We divide into two cases (I-i) $f=x^{-1}$, and (I-ii) $f\in\irr_RK[x]$.

(I-i) Assume $f=x^{-1}\in{\rm r}(\cS)$. Then $x^{-1} \in\cS$ holds by the definition of ${\rm r}(\cS)$.
In this case $k\otimes_RN^{R}_x = M_{x^{-1}}^k \in \sT_k(\cS)$ holds.
Thus $N^R_x\in\psi_\m(\sT_k(\cS))$ and hence $M_{x^{-1}}^K=K\otimes_RN_x^R \in {\rm r}_{\m 0}(\sT_k(\cS))$.

(i-ii) Assume $f\in\irr_RK[x]\cap{\rm r}(\cS)$.
Write $f(x) =r_0 + r_1x + \dots + r_d x^d$ with $r_d\neq0$.
Since $R$ is a discrete valuation ring and $f$ is primitive, we have $\overline{f}\neq0$.
%By the definition of ${\rm r}(\cS)$,
Since $W(f) \subseteq \cS$, we have
%Let $\tilde{M}^{R}_f \in\mod RQ$.
\[k\otimes_RL^{R}_f \stackrel{{\rm\ref{lem-M'}(b)}}{\in} \sT_k(W(f)) \subseteq \sT_k(\cS)\]
and hence $L^R_f\in\psi_\m(\sT_k(\cS))$.
Applying Lemma \ref{lem-M'}(a) to the canonical inclusion $R\to K$, we have $K\otimes_RL^{R}_f=M_f^K$  and hence $M_f^K\in {\rm r}_{\m 0}(\sT_k(\cS))$.

(II)
We prove ${\rm r}_{\m 0}(\sT_k(\cS)) \subseteq \sT_K({\rm r}(\cS))$. 
Fix $M\in\mod RQ$ satisfying $k\otimes_R M \in \sT_k(\cS)$.  
It suffices to prove $K\otimes_R M \in\sT_K({\rm r}(\cS))$.
We may assume that $M$ is a free $R$-module of finite rank by replacing $M$ with its quotient by the $R$-torsion submodule.
% $M/\mathrm{tors}_RM$, we may assume that $M$ is free $R$-module, where $\mathrm{tors}_RM=\{ x \in M \mid \mbox{$ax = 0 $ for some $a\in R\setminus \{0\}$} \}$.

(II-i)
We prove $K\otimes_R M \in\sT_K({\rm r}(\cS))$ in the case where $K\otimes_R M$ is indecomposable.
Assume $K\otimes_R M \not\in\sT_K({\rm r}(\cS))$.
Then either $K\otimes_R M \in \ind \cP^K$ or $K\otimes_R M \in \ind \cR^K$ holds.

We claim that there exists $N\in\mod RQ$ such that $K\otimes_R N\simeq K\otimes_R M$ and $k\otimes_R N \in (\cP^k \vee \cR^k)\setminus\sT_k(\cS)$.
Then by applying Lemma \ref{lem-XPHYT} for $(X, Y)=(N, M)$, we have $k\otimes_R M \notin \sT_k(\cS)$, a contradiction.

If $K\otimes_R M \in \ind\cP^K$, by Proposition \ref{prop-tube-k-K}, there exists $\alpha\in\Phi_{\rm pp}^{\rm rS}$ such that $K\otimes_R M\simeq M_\alpha^K$. Then $N:=M_\alpha^R$ satisfies $K\otimes_R N\simeq K\otimes_R M$ and $k\otimes_R N=M_\alpha^k\in\ind\cP^k\subset(\cP^k \vee \cR^k)\setminus\sT_k(\cS)$, as desired.
%there exists $X \in\ind\cP^k$ such that $K\otimes_k X\simeq K\otimes_R M$. Then $N:=R\otimes_k X\in\mod RQ$ satisfies  $K\otimes_R N\simeq K\otimes_R M$ and $k\otimes_R N=X$ as desired.

If $K\otimes_R M \in \ind\cR^K$, there exist $f\in \irr_R K[x]$ and $\ell\ge1$ such that $K\otimes_R M \simeq M_{f^\ell}^K$.
Then $N:=L_{f^\ell}^R\in\mod RQ$ satisfies $K\otimes_R N \simeq M_{f^\ell}^K \simeq K\otimes_R M$.
Since $K\otimes_R M \not\in\sT_K({\rm r}(\cS))$, we have $f \not\in {\rm r}(\cS)$ and $W(f^\ell)=W(f)\not\subseteq\cS$. By Lemma \ref{lem-M'}(b), we have $k\otimes_R N\in \cR^k$ and $\sT(k\otimes_R N) = \sT_k(W(f^\ell))\not\subseteq \sT_k(\cS)$. Thus $k\otimes_R N \in\cR^k\setminus \sT_k(\cS)$ holds, as desired.
%We complete the proof of the claim.

(II-ii)
We prove $K\otimes_RM \in\sT_K({\rm r}(\cS))$ in general.
% the case where $M\otimes_R K$ is decompo.
It suffices to show that each indecomposable direct summand $X$ of $K\otimes_R M$ belongs to $\sT_K({\rm r}(\cS))$.
%Let $M\otimes_R K = \bigoplus_{i=1}^{\ell}X_i$ be an indecomposable decomposition.
By \cite[Lemma 2.3(a)]{Iyama-Kimura}, there exists a factor module $N$ of $M$ such that $X\simeq K \otimes_R N$.
Since $k \otimes_RN\in \gen(k\otimes_R M) \subseteq \sT_k(\cS)$, we can apply (II-i) to $M:=N$ to get $X \simeq K \otimes_R N \in \sT_K({\rm r}(\cS))$, as desired.
%\qed
%We complete the proof.
\end{proof}

%%%%%%%%%%%%%%%
%%%%%%%%%%%%%%%
%%%%%%%%%%%%%%%
\section*{Acknowledgements}
The first author is supported by JSPS Grant-in-Aid for Scientific Research (B) 22H01113, (B) 23K22384, (C) 18K03209. The second author is supported by Grant-in-Aid for JSPS Fellows JP22J01155.
%%%%%%%%%%%%%%%
%%%%%%%%%%%%%%%
%%%%%%%%%%%%%%%

\end{document}